\documentclass[11pt]{article}
 
\usepackage{coursE}
\usepackage{url}
\usepackage{accents}

\numberwithin{equation}{section}

\renewcommand{\d}{\delta}

\DeclareMathOperator{\Sp}{Sp}

\title{Boundedness of Riesz transform on $H^1$ under sub-Gaussian estimates}

\author{{\Large Joseph {\sc Feneuil}\footnote {The author is  supported by the ANR project ``Harmonic Analysis at its Boundaries'',   ANR-12-BS01-0013-03}} \\
{Institut Fourier, Grenoble} \\
joseph.feneuil@ujf-grenoble.fr
\vspace{1cm}}
\date{\today}

\begin{document}

\maketitle

\begin{abstract}
Let $\Gamma$ be a graph and $P^k$ a Markov chain. Assume that $\Gamma$  satisfies the doubling measure property and the kernel $p_k$ of $P^k$ verifies pointwise estimates, for example pointwise sub-Gaussian estimates.
We prove of the $L^p$-boundedness of the Riesz transform $\nabla \Delta^{-\frac12}$ for any $p\in (1,2)$, and so extend the main result of \cite{CDleq2} and \cite{Russ} to spaces with fractal structures (Sierpinski gaskets, ...).

The $L^p$-boundedness of $\nabla \Delta^{-\frac12}$ is deduced from the $H^1(\Gamma)$-$H^1(T_\Gamma)$ boundedness of the Riesz transform, where $H^1(\Gamma)$ and $H^1(T_\Gamma)$ stand respectively for an Hardy space of functions and for an Hardy space of $1$-forms. The interpolation between our space $H^1(\Gamma)$ and $L^2(\Gamma)$ provides $L^p(\Gamma)$, and is therefore bigger than the Hardy space defined via Gaffney estimates (\cite{AMR}, \cite{Fen2}).

The same results holds on Riemannian manifolds and an alternative proof can be $L^p$-boundedness of the Riesz transform can be found in \cite{CCFR}.
\end{abstract}



\tableofcontents

\pagebreak

We use the following notations. $A(x) \lesssim B(x)$ means that there exists $C$  independent  of $x$ such that $A(x) \leq C \,  B(x)$ for all $x$, 
while  $A(x) \simeq B(x)$ means that $A(x) \lesssim B(x)$ and $B(x) \lesssim A(x)$.
The parameters  from which the constant is  independent  will be either obvious from context or recalled. 

\noindent Furthermore, if $E,F$ are Banach spaces, $E \subset F$ means that $E$ is continuously included in $F$. In the same way, $E=F$ means that the norms are equivalent.

\section{Introduction and statement of the results}

\subsection{Introduction}

Let $d\in \N^*$. In the Euclidean case $\R^d$, the Riesz transforms are the linear operators $\partial_j (-\Delta)^{-\frac12}$.
A remarkable property of the Riesz transforms is that they are $L^p$ bounded for all $p\in (1,+\infty)$ (see \cite[Chapter 2, Theorem 1]{Steinsingular}), which implies the equivalence
\begin{equation} \label{EquivalenceIntro} \|\nabla f\|_{L^p} := \sum_{j=1}^d \|\dr_j f\|_{L^p} \simeq \|(-\Delta)^{1/2} f\|_{L^p} \end{equation}
for any $p\in (1,+\infty)$. It allows us, when $p\in (1,+\infty)$, to characterize $W^{1,p}(\R^d)$ as the spaces of functions $f\in L^p(\R^d)$ satisfying $(-\Delta)^{1/2} f \in L^p(\R^d)$ and to extend the Sobolev spaces $W^{s,p}(\R^d)$ to the case where $s$ is not an integer. 

The $L^p$ boundedness of Riesz transforms has been extended to other settings. Let $M$ be a complete Riemannian manifold, with $\nabla$ the Riemannian gradient and $\Delta$ the Beltrami Laplace operator.
Assume $M$ is doubling. Under pointwise Gaussian upper estimate of the heat kernel $h_t$, the Riesz transform $\nabla\Delta^{-\frac12}$ is bounded on $L^p(M)$ for all $p\in (1,2]$ (see \cite{CDleq2}).
When $p>2$, the $L^p$ boundedness of the Riesz transform holds under much stronger condition, expressed in term of Poincar\'e inequalities on balls and of the domination of the gradient of the semigroup in $L^q$ for some $q>p$
(with $L^2$ Poincar\'e inequality, see \cite{CDgeq2}; with $L^q$ Poincaré inequality, see \cite{BCF}).
Similar results were established in the case of graphs (see \cite{Russ} when $p<2$, see \cite{BRuss} when $p>2$).

We are interested now by the limit case $p=1$. 
It appears that the Hardy space $H^1$ is the proper substitute of $L^1$ when Riesz transforms are involved. 
In the Euclidean case, $H^1(\R^d)$ can be defined as the space of functions $f\in L^1(\R^d)$ such that $\dr_j(-\Delta)^{-\frac12} f \in L^1(\R^d)$ for all $j\in \bb 1,d\bn$ (see \cite{FS2}).
Moreover, the Riesz transforms $\dr_j(-\Delta)^{-\frac12}$, that are bounded from $H^1(\R^d)$ to $L^1(\R^d)$, are actually bounded on $H^1(\R^d)$.

This last result, namely the $H^1$ boundedness of the Riesz transform, has been extended to complete Riemannian manifolds in \cite{AMR} - completed in \cite{AMM} - and on graphs in \cite{Fen2},
under the only assumption that the space is doubling. In order to do this, the authors introduced some Hardy spaces of functions and of forms defined by using the Laplacian.
In \cite{AMR}, when $M$ is a (complete doubling) Riemannian manifold, the authors deduced then a $H^p(M)$ boundedness of the Riesz transform for $p\in (1,2)$, 
where the spaces $H^p$ are defined by means of quadratic functionals. Under pointwise Gaussian upper estimates $H^p(M) =L^p(M)$ and thus they recover $L^p(M)$ boundedness of the Riesz transform obtained in \cite{CDleq2}.

Our goal in this paper is different to the one in \cite{AMR} and \cite{Fen2}. If we have some appropriate pointwise estimates on the Markov kernel - such as sub-Gaussian estimates - we build Hardy spaces of functions $H^1(\Gamma)$ and of 1-forms $H^1(T_\Gamma)$ such that:
\begin{itemize}
\item The Riesz transform is bounded from $H^1(\Gamma)$ to $H^1(T_\Gamma)$,
\item Any linear operator bounded from $H^1(\Gamma)$ to $L^1(\Gamma)$ and on $L^2(\Gamma)$ is also bounded on $L^p(\Gamma)$ for any $p\in (1,2)$.
\end{itemize}
In particular, the Hardy spaces defined here are bigger (or equal) than the ones in \cite{Fen2}, and the $L^p(\Gamma)$ boundedness of the Riesz transform will be obtained for a class of graphs strictly larger than the ones in \cite{Russ,Fen2}. 
All our results can be adapted to the case of Riemannian manifolds. We discuss the main differences between the two cases in Appendix \ref{Riemannian}.

Even if our goal differs from \cite{AMR}, the proofs of this article follow and extend the methods introduced in \cite{AMR} and \cite{HoMay},  which consist to build Hardy spaces using only Gaffney-type estimates. These methods have been adapted to Hardy spaces on general measured spaces in \cite{HLMMY} and on graphs in \cite{Fen2}.
Moreover, on a manifold $M$ and under sub-Gaussian estimates of the heat kernel, a Hardy space $H^1(M)$ of functions has been built in \cite{KU}, \cite{ChenThesis}. The Hardy space $H^1(\Gamma)$ in the present paper is the counterpart of this space $H^1(M)$ on graphs.

Note that a more direct proof of the $L^p$ boundedness of the Riesz transform under sub-Gaussian estimates for $p\in (1,2)$ can be now found in \cite{CCFR}. To do it, the authors proved the weak $L^1$-boundedness of the Riesz transform by using an idea of the present paper, that is a new use of the relation of Stein \cite[Lemma 2, p.49]{Stein1}. Contrary to the present paper, the article \cite{CCFR} is primarily written  in the case of Riemannian manifolds.

\subsection{Main result}

Let consider a infinite connected graph with measure $m$. We write $x \sim y$ if the two vertices $x$ and $y$ are neighbors. A finite sequence of vertices $x_0, \dots, x_n$ is a path (of length $n$) if $x_{i-1} \sim x_i$ for any $i\in \bb 1,n \bn$. The canonical distance $\d(x,y)$ is defined as the length of the shortest path linking $x$ to $y$.
We denote by $B_{\d}(x,r)$ the open ball of center $x$ and of radius $r$ and by $V_\d(x,r)$ the quantity $m(B_\d(x,r))$.

Let $P^k$ be a Markov chain - or discrete semigroup -  on $\Gamma$ and let $p_k(x,y)$ the kernel of$P^k$.  
We define a positive Laplacian by $\Delta:= I-P$, a length of the gradient by $\nabla$, the "tangent bundle" of $\Gamma$ by $T_\Gamma$ and the "external differentiation" by $d$. Complete definitions and properties are given in Section \ref{defgraphs}. 

We have from the definition and  spectral theory, for all $f \in L^2(\Gamma)$,
\[ \|d f\|_{L^2(T_\Gamma)}^2 = \|\nabla f\|_{L^2(\Gamma)}^2 =(\Delta f,f)=\left\|\Delta^{1/2} f\right\|_{L^2(\Gamma)}^2.\]
One says that the Riesz transform $\nabla\Delta^{-1/2}$ is $L^p$ bounded on $M$ if
\begin{equation} \label{Rp} \tag{$R_p$}
\|\nabla f\|_p \lesssim\left\|\Delta^{1/2} f\right\|_p, \,\,\forall f \in \mathcal C_0^{\infty}(M).
\end{equation}

\medskip

The first condition we need on the graph is the following one
\begin{defi}
There exists $\epsilon:=\epsilon_{LB}$ such that, for any $x\in \Gamma$,
\begin{equation} \label{LB} \tag{LB}
(P\1_{x})(x) > \epsilon.
\end{equation}
\end{defi}

The condition \eqref{LB} mean that the probability to stay on one point is uniformly bounded from below. It implies in particular the following estimate on $P^k$:
\begin{equation} \label{Ana} \|\Delta P^k f\|_{L^2} \leq C \frac1k\|f\|_{L^2}. \end{equation}
Note that a discrete semigroup $P^k$ satisfying \eqref{Ana} is said analytic (see \cite{CSC}). The converse implication \eqref{Ana} to \eqref{LB} doesn't hold. But, under a doubling volume property, \eqref{Ana} forces an odd power of $P$ to satisfy \eqref{LB} (see \cite{FenLB}).

\medskip

Our article aims to extend the following statement for the Riesz transform on graphs (see \cite[Theorem 1.40 and Remark 1.41a]{Fen2}, see  \cite{AMR} for the counterpart on Riemannian manifolds, see also \cite{CDleq2} and \cite{Russ}). 

\begin{theo} \label{thm-cd}
Let $\Gamma$ be a graph satisfying the volume doubling property
\begin{equation}\label{DV2}
V_\d(x,2r)\lesssim V_\d(x,r) \qquad \forall x\in M, \, \forall r\in \N.\tag{D$_2$}.
\end{equation}
Then, the Riesz transform $d\Delta^{-\frac12}$ is bounded from $H^1_{\d^2}(\Gamma)$ to $H^1_{\d^2}(T_\Gamma)$, where $H^1_{\d^2}(\Gamma)$ and $H^1_{\d^2}(T_\Gamma)$ stand for some complete spaces included in respectively $L^1(\Gamma)$ and $L^1(T_\Gamma)$.

Assume moreover that the kernel $p_k$ of $P^k$ satisfies the pointwise upper estimate 
\begin{equation} \label{UEd2} \tag{UE$_2$}
p_k(x,y)\lesssim \frac{1}{V_\d(x,\sqrt k)} \exp\left(-c\frac{\d^2(x,y)}{k}\right),\,\,\forall x,y\in M,\, t>0,
\end{equation}
for some $c>0$.
Then the Riesz transform $\nabla \Delta^{-\frac12}$ is bounded on $L^p(\Gamma)$ for any $p\in (1,2)$.
\end{theo}


Note that in particular, we didn't know previously any case of graphs where the $L^p$ boundedness of Riesz transform holds for any $p\in (1,2)$ and where the pointwise Gaussian estimate \eqref{UEd2} doesn't hold.

\medskip

We want to define some condition on kernel $p_k(x,y)$ weaker than the pointwise Gaussian estimates. Fix $\beta$ a function on $M^2$, $1\leq \beta \leq B < +\infty$. Define $\rho(x,y) = \d(x,y)^{\beta(x,y)}$ and $V_\rho(x,k) = m(\{y\in M, \, \rho(x,y) < k\})$.

\begin{defi}
We say that $M$ satisfies \eqref{UE} if for any $N\in \N$, there exists $C_N >0$ such that
\begin{equation} \label{UE} \tag{UE$_\beta$}
p_{k-1}(x,y)\leq \frac{C_N}{V_\rho(x,k)} \left(1+ \frac{\rho(x,y)}{k} \right)^{-N} \forall x,y\in M,\, k\geq 1.
\end{equation}
\end{defi}

Check that the condition \eqref{UEd2} implies \eqref{UE} when $\beta \equiv 2$. We let the reader see Appendix \ref{Fen4SectionExamples} to see examples of graphs satisfying \eqref{UE} for various functions $\beta$ and references on these pointwise (non necessary Gaussian) estimates.

\noindent Remark that $p_{k-1}(x,y)$ can be replaced equivalently by $p_{k}(x,y)$ in \eqref{UE}. Yet the shift in $k$ allow us to give an estimate of $p_n(x,y)$ for any $n\geq 0$.

\begin{defi}
We say that $M$ satisfies \eqref{DV} if 
\begin{equation}\label{DV}
V_\rho(x,2r)\lesssim V_\rho(x,r) \qquad \forall x\in M, \, \forall r>0.\tag{D$_\beta$}
\end{equation}
\end{defi}

Note that when $\beta$ is a constant function, then \eqref{DV} is equivalent to \eqref{DV2}. Note also that the assumption \eqref{DV} yields

\begin{prop} \label{Fen4propDV}
 Let $\Gamma$ satisfying  \eqref{DV}.  Then there exists $d>0$ such that
\begin{equation} \label{Fen4PDV} V_\rho(x,\lambda r) \lesssim \lambda^d V_\rho(x,r) \qquad \forall x\in \Gamma, \, r>0 \text{ and } \lambda \geq 1.\end{equation}
\end{prop}

Our main result is

\begin{theo} \label{MainTheo}
Let $(\Gamma,m)$ be a connected graph and $P$ a random walk on $\Gamma$. Assume that $\Gamma$ satisfies \eqref{DV} and \eqref{UE}.
Then there exists two complete spaces $H^1(\Gamma) \subset L^1(\Gamma)$ and $H^1(T_\Gamma) \subset L^1(T_\Gamma)$ such that:
\begin{enumerate}
\item The Riesz transform $d\Delta^{-\frac12}$ is an homomorphism between $H^1(\Gamma)$ and $H^1(T_\Gamma)$.
\item Every linear operator bounded from $H^1(\Gamma)$ to $L^1(\Gamma)$ and bounded on $L^2(\Gamma)$ is bounded on $L^p(\Gamma)$ for any $p\in (1,2)$.
\end{enumerate}
As a consequence, the Riesz transform $\nabla\Delta^{-\frac12}$ is bounded on $L^p(M)$ for any $p\in (1,2)$.
\end{theo}

\begin{rmk}
Actually, the point (2) of the above Theorem can be extended to linearizable operators. The linearizable operators are a subclass of the sublinear operator containing the quadratic functionals and the maximal operators, see \cite[Definition 5.2]{BZ}.
\end{rmk}

The plan of the paper is as follows.  
Section \ref{Tools} introduce the discrete setting and the main notations. Moreover, we introduced the tools used in the paper, such as the decomposition in atoms in tents spaces or the interpolation with Hardy spaces. 
Section \ref{Estimates} provides the estimates needed in the sequel of the article. In particular, we give some crucial off-diagonal estimates on the gradient of the Markov chain $P^k$, by using a relation of Stein (\cite[Chapter 2, Lemma 2]{Stein1}) and it adaptation to the discrete case (\cite{Dungey2}, \cite[Section 4]{Fen1}).
In Section \ref{EqualityHardySpaces}, we give some equivalent definitions of the Hardy spaces of functions and of 1-forms.  
In Section \ref{ProofMainTheo}, we prove the Theorem \ref{MainTheo}. More exactly, we will check that the Hardy spaces previously constructed in Section \ref{EqualityHardySpaces} satisfies the assumption of Theorem \ref{MainTheo}.
In Appendix \ref{Fen4SectionExamples}, you will find examples of manifolds satisfying \eqref{UE} for various functions $\beta$. We prove in this section that there exists a graph that satisfies \eqref{UE} for some non constant $\beta$ and that doesn't satisfy \eqref{UE} for any $\beta \equiv m$ constant.
At last, Appendix \ref{Riemannian} is devoted to the statement of the counterpart results in the case of Riemannian manifolds. We discuss there few differences with the discrete case.

\section{Notations and tools}

\label{Tools}

\subsection{The discrete setting}

\label{defgraphs}

Let $\Gamma$ be an infinite set and $\mu_{xy} = \mu_{yx} \geq 0$ a symmetric weight on $\Gamma \times \Gamma$. 
The couple $(\Gamma, \mu)$ induces a (weighted unoriented) graph structure if we define the set of edges by
$$E = \{ (x,y) \in \Gamma \times \Gamma, \, \mu_{xy} >0 \}.$$
We call then $x$ and $y$ neighbors (or $x\sim y$) if  and only if  $(x,y) \in E$.\par
\noindent We will assume that the graph is locally uniformly finite,  that is  there exists $M_0 \in \N$ such that for all $x\in \Gamma$, 
 \begin{equation} \label{Fen4defM0}
\# \{y\in \Gamma, \, y\sim x\} \leq M_0.
\end{equation}
In other words, the number of neighbors of a vertex is uniformly bounded. Note that a graph that satisfies \eqref{DV2} or \eqref{DV} is necessary locally uniformly locally finite \par

\noindent  We define the weight $m(x)$ of a vertex $x \in \Gamma$ by $m(x) = \sum_{x\sim y} \mu_{xy}$. 
More generally, the volume (or measure) of a subset $E \subset \Gamma$ is defined as $m(E) := \sum_{x\in E} m(x)$. \par
\noindent We define now the $L^p(\Gamma)$ spaces. For all $1\leq p < +\infty$, we say that a function $f$ on $\Gamma$ belongs to $L^p(\Gamma,m)$ (or $L^p(\Gamma)$) if
$$\|f\|_p := \left( \sum_{x\in \Gamma} |f(x)|^p m(x) \right)^{\frac{1}{p}} < +\infty,$$
while $L^\infty(\Gamma)$ is the  space  of functions satisfying 
$$\|f\|_{\infty} : = \sup_{x\in\Gamma} |f(x)| <+\infty.$$
Let us define for all $x,y\in \Gamma$ the discrete-time reversible Markov kernel $p$ associated with the measure $m$ by $p(x,y) = \frac{\mu_{xy}}{m(x)m(y)}$.
The discrete kernel $p_k(x,y)$ is then defined recursively for all $k\geq 0$ by
\begin{equation}
\left\{ 
\begin{array}{l}
p_0(x,y) = \frac{\delta(x,y)}{m(y)} \\
p_{k+1}(x,y) = \sum_{z\in \Gamma} p(x,z)p_k(z,y)m(z).
\end{array}
\right.
\end{equation}

Notice that for all $k\geq 1$, we have
\begin{equation} \label{Fen4sum=1}
 \|p_k(x,.)\|_{L^1(\Gamma)} = \sum_{y\in \Gamma} p_k(x,y)m(y) = \sum_{d(x,y) \leq l} p_k(x,y)m(y) = 1 \qquad \forall x\in \Gamma,
\end{equation}
and that the kernel is symmetric:
\begin{equation} \label{Fen4symmetry}
 p_k(x,y) = p_k(y,x) \qquad \forall x,y\in \Gamma.
\end{equation}
For all functions $f$ on $\Gamma$, we define $P$ as the operator with kernel $p$, i.e.
\begin{equation}\label{Fen4defP} Pf(x) = \sum_{y\in \Gamma} p(x,y)f(y)m(y) \qquad \forall x\in \Gamma.\end{equation}
Note that in this case, the assumption \eqref{LB} becomes
\begin{equation} \tag{LB}
p(x,x)m(x) > \epsilon_{LB} \qquad \forall x\in \Gamma
\end{equation}
Check also that $P^k$ is the operator with kernel $p_k$. 

Since $p(x,y) \geq 0$ and \eqref{Fen4sum=1} holds, one has, for all $p\in [1,+\infty]$ ,
\begin{equation} \label{Fen4Pcont}
 \|P\|_{p\to p } \leq 1.
\end{equation}

\begin{rmk} \label{Fen4poweri-p}
Let $1\leq p<+\infty$. Since, for all $k\geq 0$, $\left\Vert P^k\right\Vert_{p\rightarrow p}\leq 1$, the operators $(I-P)^{\beta}$ and $(I+P)^{\beta}$ are $L^p$-bounded for all $\beta\geq 0$ (see \cite{CSC}).
\end{rmk}

\noindent We define a nonnegative Laplacian on $\Gamma$ by $\Delta = I-P$. One has then

\begin{equation}
 \begin{split}
 \left<\Delta f,f\right>_{L^2(\Gamma)} & = \sum_{x,y\in \Gamma} p(x,y)(f(x)-f(y))f(x)m(x)m(y) \\
& = \frac{1}{2} \sum_{x,y\in \Gamma} p(x,y)|f(x)-f(y)|^2m(x)m(y),
 \end{split}
\end{equation}
where we use \eqref{Fen4sum=1} for the first equality and \eqref{Fen4symmetry} for the second one. The last calculus proves that the following operator
$$\nabla f(x) = \left( \frac{1}{2} \sum_{y\in \Gamma} p(x,y) |f(y)-f(x)|^2 m(y)\right)^{\frac{1}{2}},$$
called ``length of the gradient'' (and the definition of which is taken from \cite{CoulGrigor}), satisfies
\begin{equation} \label{Fen4Katopb}
\|\nabla f\|^2_{L^2(\Gamma)} =  <\Delta f,f>_{L^2(\Gamma)} =  \|\Delta^{\frac12}f\|_{L^2(\Gamma)}^2.
\end{equation}

\noindent We recall now definitions of $1$-forms on graphs and their first properties (based on \cite{Fen2}).
We define, for all $x\in \Gamma$, the set $T_x = \{(x,y)\in \Gamma^2, \, y\sim x\}$ and for all set $E\subset \Gamma$,
$$T_E = \bigcup_{x\in E} T_x = \{(x,y)\in E \times \Gamma, \, y\sim x\}.$$

\begin{defi}
If $x\in \Gamma$, we define, for all  functions  $F_x$ defined on $T_x$ the norm 
$$|F_x|_{T_x} = \left( \frac{1}{2} \sum_{y\sim x} p(x,y) m(y) |F_x(x,y)|^2 \right)^\frac12.$$
Moreover, a function $F: \, T_\Gamma \to \R$ belongs to $L^p(T_\Gamma)$ if
\begin{enumerate}[(i)]
 \item $F$ is antisymmetric, that is $F(x,y) = -F(y,x)$ for all $x\sim y$,
 \item $\|F\|_{L^p(T_\Gamma)} := \left\|x\mapsto |F(x,.)|_{T_x} \right\|_{L^p(\Gamma)} < +\infty$.
\end{enumerate}
The Hilbert space $L^2(T_\Gamma)$ is outfitted with the inner product  $\langle\cdot,\cdot\rangle$  defined as
$$\left<F,G\right> = \frac{1}{2} \sum_{x,y\in \Gamma} p(x,y) F(x,y) G(x,y) m(x) m(y).$$
\end{defi}

\begin{defi} Let $f$ a function on $\Gamma$ and $F$ an antisymmetric function on $T_\Gamma$. 
Define the operators $d$ and $d^*$ by
$$d f(x,y) : = f(x) - f(y) \qquad \forall (x,y)\in T_\Gamma$$
and
$$d^* F(x) : = \sum_{y\sim x} p(x,y) F(x,y) m(y) \qquad \forall x\in \Gamma.$$ 
\end{defi}

\begin{rmk}
It is plain to see that $d^*d = \Delta$ and $|df(x,.)|_{T_x} = \nabla f(x)$. 

\noindent  As the notation $d^{\ast}$ suggests,  $d^*$ is the adjoint of $d$, that is for all $f\in L^2(\Gamma)$ and $G\in L^2(T_\Gamma)$, 
\begin{equation} \label{Fen4dstaradjointofd}\left<d f,G \right>_{L^2(T_\Gamma)} = \left< f, d^*G\right>_{L^2(\Gamma)}.\end{equation}
The proof of this fact can be found in \cite[Section 8.1]{BRuss}.
\end{rmk}

We introduce a  subspace  of $L^2(T_\Gamma)$, called $H^2(T_\Gamma)$, defined as the  closure  in $L^2(T_\Gamma)$  of
$$ E^2(T_\Gamma) : = \{F \in L^2(T_\Gamma) , \, \exists f\in L^2(\Gamma): \, F = df\}.$$
Notice that $d\Delta^{-1}d^* = Id_{E^2(T_\Gamma)}$  (see \cite{Fen2}).  The functional $d\Delta^{-1}d^*$ can be extended to a bounded operator on $H^2(T_\Gamma)$ and 
\begin{equation} \label{Fen4dDeltadstar}
d\Delta^{-1}d^* = Id_{H^2(T_\Gamma)}.
\end{equation}

\noindent Let us  recall Proposition 1.32 in  \cite{Fen2} .

\begin{prop} \label{Fen4danddstar}
 For all $p\in [1,+\infty]$,  the operator $d^*$ is bounded from $L^p(T_\Gamma)$ to $L^p(\Gamma)$.\par
\noindent The operator $d\Delta^{-\frac12}$ is an isometry from $L^2(\Gamma)$ to $H^2(T_\Gamma)$, 
and the operator $\Delta^{-\frac12}d^*$ is an isometry from $H^2(T_\Gamma)$ to $L^2(\Gamma)$.
\end{prop}

\subsection{Metric}

\label{metric}

As said in the introduction, the distance $\d(x,y)$ is defined as the length of the shortest path linking $x$ to $y$.
In our computations, we will suppose that there exists a function $\beta$ (bounded from below by 1 and from above by $B$) such that \eqref{DV} and \eqref{UE} holds on $\Gamma$.
Fix $\rho:=\d^\beta$. Check that we have 
\begin{equation} \label{Crho} \rho(x,z) \leq 2^{B-1} \left[\rho(x,y) + \rho(y,z)\right],\end{equation}
that is $\rho$ is a quasidistance on $\Gamma$. 

The metric used in the sequel will be the one of $\rho:=\d^\beta$.
More precisely, when $k\in \N^*$ and $x\in \Gamma$, the balls $B(x,k)$ denote the sets $\{y\in \Gamma, \, \rho(x,y) < k\}$, the notation $V(x,k)$ is used for $m(B(x,k))$. Besides, the set $C_0(x,k)$ denotes $B(x,2^{B+1})$ and when $j\geq 1$, the sets $C_j(x,k)$ denote  the annuli $B(x,2^{B+j+1} k) \setminus B(x,2^{B+j} k)$. The important point is to remark that
$$\rho(B(x,k),C_j(x,k)) \geq 2^j k.$$
Indeed, for any $y\in B(x,k)$ and $z\in C_j(x,k)$, we have
$$2^{B+j}k \leq \rho(x,z) \leq 2^{B-1} \left[\rho(x,y) + \rho(y,z)\right] \leq 2^{B-1}k + 2^{B-1} \rho(y,z)$$
and thus
$$\rho(y,z) \geq \frac{2^{B+j} - 2^{B-1}}{2^{B-1}} k \geq 2^j k.$$

\subsection{Pointwise and Gaffney estimates for $p_k(x,y)$.}

\begin{prop} \label{higherpointwiseestimate}
Let $(\Gamma, \mu)$ satisfying \eqref{LB}, \eqref{UE} and \eqref{DV}. Then for all $j\in \N$ and all $N\in \N$, there exists $C_{j,N}$ such that 
$$|D(1)^jp_{k-1}(x,y)| \leq \frac{C_{j,N}}{k^j V(x,k)} \left( 1 + \frac{\rho(x,y)}{k} \right)^{-N},$$
where $D(1)$ stands for the operator on sequences:
$$D(1)p_{k-1} = p_k - p_{k-1}.$$
\end{prop}

\begin{proof}
The proof of this result can be done with the same method as \cite[Theorem 1.1]{Dungey} (see also \cite[Theorem A.1]{Fen1}).
\end{proof}

We can deduce the following $L^1$-$L^2$ Gaffney estimates.

\begin{cor} \label{Fen4LqLpforDeltaPk}
Let $(\Gamma, \mu)$ satisfying \eqref{LB}, \eqref{UE} and \eqref{DV}. Let $j\in \N$ and $N\in \N$.

If the sets $E,F\subset \Gamma$, $x_0\in \Gamma$ and $k\in \N$ satisfies one of the following condition
\begin{enumerate}
\item $\sup\{\rho(x_0,y), \, y\in F\} \leq 2^B \rho(E,F)$,
\item $\sup\{\rho(x_0,x), \, x\in E\} \leq 2^B \rho(E,F)$,
\item $\sup\{\rho(x_0,y), \, y\in F\} \leq k$,
\item $\sup\{\rho(x_0,x), \, x\in E\} \leq k$,
\end{enumerate}
then for any $f\in L^1(\Gamma)$,
$$\|(k\Delta)^j P^{k-1} [f\1_F] \|_{L^2(E)} \leq \frac{C_{j,N}}{k^j V(x_0,k)} \left( 1 + \frac{\rho(E,F)}{k} \right)^{-N} \|f\|_{L^1},$$
where the constant $C_{j,N}$ depends only on $j,N$ and $\Gamma$.
\end{cor}

\begin{proof}
The proof of this result is similar to the one of \cite[Theorem A.3]{Fen1}.
\end{proof}

\subsection{Tent spaces}

\label{sectionTent}

In all this section, $(\Gamma, \mu)$ is a weighted graph outfitted with a quasidistance $\rho=\d^\beta$ satisfying \eqref{DV}. 

We state in this section the atomic decomposition the Tent space $T^1(\Gamma)$. The tents spaces and their atomic decomposition were introduced in the case Euclidean by Coifman, Meyer and Stein in \cite{CMS}. However the methods used in \cite{CMS} are not specific to $\R^d$ and the definition of Tent spaces - as well as their properties - can be extended to spaces with the doubling volume property. The case of Riemannian manifolds has been done by Russ in \cite{Russ2}. Some definitions and results in the case of graphs (endowed with a quasidistance) are given here, but the proofs - similar to the ones of \cite{CMS} and \cite{Russ2} - are avoided. They can be found yet in \cite[Section D.3]{FenThesis}.

\begin{defi}
We introduce the following sets in $\Gamma \times \N^*$. If $x\in \Gamma$,
$$\gamma(x) = \{(y,k) \in \Gamma\times \N^*, \, \rho(x,y) < k\},$$
if $F\subset \Gamma$,
$$\check F = \bigcup_{x\in F} \gamma(x)$$
and if $O \subset \Gamma$, 
$$\hat O = \{(y,k) \in \Gamma \times \N^*, \, \rho(y,O^c) \geq k\}.$$
We define the functionals $\mathcal A$ and $\mathcal C$ mapping functions on $\Gamma \times \N^*$ into functions on $\Gamma$ by
$$\mathcal A f(x) = \left( \sum_{(y,k) \in \gamma(x)} \frac{m(y)}{kV(x,k)} |f(y,k)|^2 \right)^\frac{1}{2}$$
and
$$\mathcal C f(x)= \sup_{x\in B} \left(\frac{1}{V(B)} \sum_{(y,k)\in \hat B}\frac{m(y)}{k}|f(y,k)|^2 \right)^\frac{1}{2}.$$
For any $p\in [1,+\infty)$, the tent space $T^p(\Gamma)$ denotes the space of functions $f$ on $\Gamma\times \N^*$ such that $\mathcal A f \in L^p(\Gamma)$.
Moreover, the tent space $T^\infty(\Gamma)$ is the space of functions $f$ on $\Gamma\times \N^*$ such that $\mathcal C f \in L^\infty(\Gamma)$.
The tent space $T^p(\Gamma)$ is equipped with the norm $\|f\|_{T^p} = \|\mathcal A f\|_{L^p}$ (or $\|f\|_{T^\infty} = \|\mathcal C f\|_{L^\infty}$ when $p=\infty$).
\end{defi}

\begin{rmk}
 One has the following equality of sets
$$\check{O^c} = (\hat O)^c.$$
\end{rmk}

\begin{defi}
 A function $A$ defined on $\Gamma \times \N^*$ is a $T^1$-atom if there exists a ball $B$ such that
\begin{enumerate}[(i)]
 \item $A$ is supported in $\hat B$,
 \item $\sum_{(y,k)\in \hat B} \frac{m(y)}{k} |A(y,k)|^2 \leq \frac{1}{V(B)}$.
\end{enumerate}
\end{defi}

\begin{lem}
 There exists a constant $C>0$ such that, for every $T^1$-atom $A$, one has
$$\|A\|_{T^1} \leq C $$
and 
$$\|A\|_{T^2} \leq \frac{C}{V(B)^\frac12}$$
\end{lem}

\begin{theo} \label{theoTent}
 \begin{enumerate}[(i)]
  \item The following inequality holds, whenever $f\in T^1(\Gamma)$ and $g\in T^\infty(\Gamma)$:
$$\sum_{(y,k)\in \Gamma \times \N^*} \frac{m(y)}{k} |f(y,k)g(y,k)| \lesssim \sum_{y\in \Gamma} \mathcal Af(x) \mathcal Cf(x).$$
  \item The pairing 
$$\left<f,g\right> = \sum_{(y,k)\in \Gamma \times \N^*} \frac{m(y)}{k} |f(y,k)g(y,k)|$$
realizes $T^\infty(\Gamma)$ as equivalent to the Banach space dual of $T^1(\Gamma)$.
  \item Every element $f\in T^1(\Gamma)$ can be written as $f=\sum \lambda_j a_j$ where $a_j$ are $T^1$-atoms, $\lambda_j \in \R$ and $\sum |\lambda_j| \lesssim \|f\|_{T^1}$.
  \item Moreover, if $f\in T^1(\Gamma) \cap T^2(\Gamma)$, the atomic decomposition can be chosen to be convergent in $T^2(\Gamma)$.
 \end{enumerate}
\end{theo}

\subsection{Interpolation}

\label{Fen4InterpolationSubsection}

Let us recall first a result on $L^p$ boundedness of Calder\`on-Zygmund operators (originally due to Blunck and Kunstmann, see Theorem 1.1 in \cite{BK}, see also Theorem 1.1 in \cite{Auscher2007}).

\begin{defi}
A function $f$ on $\Gamma$ is in $L^{1,\infty}(\Gamma)$ if 
$$\|f\|_{L^{1,\infty}}: = \sup_{\lambda >0} \lambda m(\{x\in \Gamma, \, |f(x)|>\lambda \}) <+\infty.$$
\end{defi}

\begin{theo} \label{Fen4BKtheoremforinterpolation}
For any ball $B$, let $A_B$ be a linear operator in $L^2(\Gamma)$.
Let $T$ be a $L^2$-bounded sublinear operator such that for all balls $B=B(x,k)$ and all functions $f$ supported in $B$
$$\frac{1}{V(x,2^{j}k)^\frac{1}{2}} \left\|T(I-A_B) f \right\|_{L^2(C_j(x,k))} \leq \alpha_j(B) \frac{1}{V(x,k)} \|f\|_{L^1(B)}$$
for all $j\geq 1$ and
$$\frac{1}{V(x,2^{j}k)^\frac{1}{2}} \left\|A_B f \right\|_{L^2(C_j(x,k))} \leq \alpha_j(B) \frac{1}{V(x,k)} \|f\|_{L^1(B)}$$
for all $j\geq 0$.

If the coefficients $\alpha_j(B)$ satisfy
$$\sup_{B=B(x,k) \text{ ball}} \sum_{j\geq 0} \frac{V(x,2^{j+1}k)}{V(x,k)} \alpha_j(B) < +\infty$$
then there exists a constant $C$ such that
$$\|Tf\|_{L^{1,\infty}} \leq C \|f\|_{L^1} \qquad \forall f\in L^2\cap L^1.$$
So by interpolation, for all $p\in (1,2]$, there exists a constant $C=C_p$ such that
$$\|Tf\|_{L^{p}} \leq C_p \|f\|_{L^p} \qquad \forall f\in L^2\cap L^p.$$
\end{theo}

Our second result deals with interpolation of Hardy spaces. We reformulate here some results of \cite{BZ} in our context.

\begin{defi}
 A function $a\in L^2(\Gamma)$ is called an atom if there exist $x\in \Gamma$ and $k\in \N^*$ and a function $b \in L^2(\Gamma)$ supported in $B(x,k)$ such that
\begin{enumerate}[(i)]
 \item $a = (I-P^k) b$,
 \item $\ds \|b\|_{L^2}  = \|b\|_{L^2(B(x,k))} \leq  V(x,k)^{-\frac{1}{2}}$.
\end{enumerate}
We say that $f$ belongs to $E^1_{0}(\Gamma)$ if $f$ admits a finite atomic representation, 
that is if there exist a finite sequence $(\lambda_i)_{i=0..N}$ and a finite sequence $(a_i)_{i=0..N}$ of atoms such that
\begin{equation} \label{Fen4sumf0}
f=\sum_{i=0}^N \lambda_i a_i.
\end{equation}
The space is outfitted with the norm
$$\|f\|_{E^1_{0}} = \inf\left\{ \sum_{j=0}^\infty |\lambda_j|, \ \sum_{j=0}^\infty \lambda_j a_j, \text{ is a finite atomic representation of $f$} \right\}.$$
\end{defi}

\begin{theo} \label{Fen4interpolationH1L2}
Let $(\Gamma,\mu,\rho)$ satisfying \eqref{LB}, \eqref{DV} and \eqref{UE}.
If $T$ is an $L^2(\Gamma)$ bounded linear operator and if there exists $C>0$ such that for all atoms
$$\|Ta\|_{L^1} \leq C,$$
then for all $p\in (1,2]$, there exists a constant $C=C(p)>0$ such that
$$\|Tf\|_p \leq C_p \|f\|_p \qquad \forall f\in L^p\cap L^2.$$
\end{theo}

\noindent The next result is an immediate corollary of Theorem \ref{Fen4interpolationH1L2}.

\begin{cor} \label{Fen4interpolationH1L2bis}
Let $(\Gamma,\mu,\rho)$ satisfying \eqref{LB}, \eqref{DV} and \eqref{UE}. Let $(H^1_0(\Gamma),\|.\|_{H^1_0})$ a normed vector space that satisfies the continuous embedding
$$E^1_0(\Gamma) \subset H^1_0(\Gamma) \subset L^1(\Gamma).$$
If $T$ is an $L^2(\Gamma)$-bounded linear operator that verifies
$$\|Tf\|_{L^1} \lesssim \|f\|_{H^1_0} \qquad \forall f\in H^1_0(\Gamma) \cap L^2(\Gamma),$$
then for all $p\in (1,2]$, there exists a constant $C=C(p)>0$ such that
$$\|Tf\|_p \leq C_p \|f\|_p \qquad \forall f\in L^p\cap L^2.$$
\end{cor}

\begin{rmk}
This corollary will be used for the Hardy space $H^1_0 = H^1(\Gamma)$ that will be defined in Section \ref{EqualityHardySpaces}.
\end{rmk}

\begin{proof} (Theorem \ref{Fen4interpolationH1L2})
Theorem \ref{Fen4interpolationH1L2} can be seen as an application of Theorem 5.3 in \cite{BZ}. This Theorem of Bernicot and Zhao provides all the conclusions of Theorem \ref{Fen4interpolationH1L2} once we checked the assumption
\begin{equation} \label{Fen4LinftyL2forPk} \sup_{y\in B(x,k)} |P^k h(y)| \lesssim C_M \inf_{z\in B(x,k)}\left[\mathcal M(|h|^2)(z)\right]^\frac12,\end{equation}
where $\mathcal M$ stands for the uncentered maximal function of Littlewood-paley defined by
$$\mathcal M f(x) = \sup_{B\ni x} \sum_{y\in B} |f(y)| m(y).$$

Yet, with Corollary \ref{Fen4LqLpforDeltaPk}, there holds for any $z\in C_0(x,k) \supset B(x,k)$,
\begin{equation}\label{PkLinftyL2}\begin{split}
   \sup_{y\in B(x,k)} |P^k h(y)| &  \leq \sum_{j\geq 0} \|P^k[h\1_{C_j(x,k)}]\|_{L^\infty(B(x,k))} \\
& \lesssim \sum_{j\geq 0} \sup{y \in B(x,k)} \sum_{z\in C_j(x,k)} p_k(y,z) h(z) m(z)\\
& \lesssim \sum_{j\geq 0} \|h\|_{L^2(C_j(x,k))} \sup{y \in B(x,k)} \|p_k(x,.)\|_{L^2(C_j(x,k))} \\
& \lesssim \sum_{j\geq 0} \|h\|_{L^2(B(x,2^{B+j+1}k))} \sup{y \in B(x,k)} \|p_k(x,.)\|_{L^2(C_j(x,k))} \\
  \end{split}\end{equation}
If $j=0$, one has for any $y\in B(x,k)$, one has then with \eqref{DV}
\[\begin{split}
\sum_{z\in C_0(x,k)} |p_k(y,z)|^2 m(z) & \lesssim \sum_{z\in C_0(x,k)} \frac{1}{V(x,k)^2} m(z) \lesssim \frac{1}{V(x,2^{B+1}k)}.
\end{split}\]
And if $j\geq 1$, $\rho(B(x,k), C_j(x,k)) \gtrsim 2^j$ and with Proposition \ref{Fen4propDV}
\[\begin{split}
\sum_{z\in C_j(x,k)} |p_k(y,z)|^2 m(z) & \lesssim \sum_{z\in C_j(x,k)} \frac{1}{V(x,k)^2} 2^{-2j(d+1)} m(z) \lesssim \frac{2^{-2j}}{V(x,2^{B+j+1}k)}.
\end{split}\]
Reinjecting these estimates on $\ds \sup{y \in B(x,k)}\|p_k(x,.)\|_{L^2(C_j(x,k))}$ in \eqref{PkLinftyL2}, one has
\[\begin{split}
   \sup_{y\in B(x,k)} |P^k h(y)|
& \lesssim \sum_{j\geq 0} \|h\|_{L^2(B(x,2^{B+j+1}k))} \frac{2^{-2j}}{V(x,2^{B+j+1}k)}  \\
& \lesssim \sum_{j\geq 0} 2^{-2j} \left( \mathcal M(|h|^2)(x) \right)^\frac12 \\
& \lesssim \left( \mathcal M(|h|^2)(x) \right)^\frac12.
  \end{split}\]
\end{proof}

\section{Off-diagonal estimates}

\label{Estimates}

In this section, $(\Gamma, \mu)$ is a weighted graph as defined in Section \ref{defgraphs} verifying \eqref{LB}. We assume that there exists a function $\beta$ bounded from below by 1 and from above by $B$ such that $\Gamma$ satisfies \eqref{DV} and \eqref{UE}. Under this circumstances, $\rho$ denotes the quasidistance $\d^\beta$ and the metric considered is the one of $\rho$ (see Section \ref{metric}).

\subsection{Gaffney estimates, first results}

\begin{defi}
We say that a family of operators $(A_k)_{k}$ satisfies  $L^p$ Gaffney estimates if
for any $N\in \N$, there exist $C_N$ such that, for any sets $E,F\subset \Gamma$ and any  function  $f\in L^p(\Gamma,m)$,
\begin{equation} \label{Fen4DGEdef} \|A_k[f\1_{F}]\|_{L^p(E)} \leq C_N \left(1+ \frac{\rho(E,F)}{k}\right)^{-N} \|f\|_p.\end{equation}
\end{defi}
 \noindent It is plain to observe that \eqref{Fen4DGEdef} is equivalent to
$$
\|A_k(f)\|_{L^p(E)} \leq C_N \left(1+\frac{\rho(E,F)}{k}\right)^{-N} \|f\|_p
$$
whenever $f$ is supported in $F$.  \par

\begin{prop} \label{GEforPk}
For any $j\in \N$ and any $p\in [1,+\infty]$, the family $([k\Delta]^jP^{k-1})_{k\in \N^*}$ satisfies $L^p$ Gaffney estimates.
\end{prop}

\begin{proof}
We will prove the cases $p=1$ and $p=+\infty$. The  conclusion  can be then deduced from these endpoint estimates  by interpolation.

Fix $N\in \N$. Let $E,F\subset \Gamma$ and $f\in L^1(\Gamma)$ and $k\geq 1$.
Since $P^k$ is uniformly bounded on $L^1(\Gamma)$ and $L^\infty(\Gamma)$, we can assume without loss of generality that $\rho(E,F) \geq k$ and $N\geq 1$.

We begin with $p=1$. 
\[\begin{split}
\|[k\Delta]^jP^{k-1}(f\1_F)\|_{L^1(E)} & = \sum_{x\in E} m(x) \left| \sum_{z\in F} D(1)^jp_{k-1}(x,z) f(z) m(z) \right| \\
& \leq \sum_{z\in F} |f(z)| m(z) \sum_{x\in E} |D(1)^jp_{k-1}(x,z)| m(x) \\
& \lesssim \sum_{z\in F}   |f(z)| m(z) \sum_{x\in E} \frac{m(x)}{V(z,k)}\left(1+ \frac{\rho(x,z)}{k}\right)^{-N-d} \\
& \quad \leq \sum_{z\in F}   |f(z)| m(z) \sum_{j\geq 0} \sum_{ \frac{\rho(z,x)}{\rho(z,E)} \in [2^j, 2^{j+1})} \frac{m(x)}{V(z,k)} \left(1+ \frac{2^j\rho(z,E)}{k}\right)^{-N-d} \\
& \quad \leq \sum_{z\in F}   |f(z)| m(z) \sum_{j\geq 0} \frac{V(z,2^{j+1}\rho(z,E))}{V(z,k)} \left(\frac{2^j\rho(z,E)}{k}\right)^{-N-d} \\ 
& \lesssim \sum_{z\in F}  |f(z)| m(z) \sum_{j\geq 0} \left( 1+ \frac{2^j\rho(z,E)}{k}\right)^{d}\left(1+ \frac{2^j\rho(z,E)}{k}\right)^{-N-d}\\
& \lesssim \sum_{z\in F}  |f(z)| m(z) \sum_{j\geq 0} \left(\frac{2^j\rho(F,E)}{k}\right)^{-N} \\
& \lesssim \|f\|_{L^1} \left(\frac{\rho(F,E)}{k}\right)^{-N} \sum_{j\geq 0} 2^{-jN} \\
& \lesssim  \|f\|_{L^1} \left(1+\frac{\rho(F,E)}{k}\right)^{-N},  \end{split}\]
where the third line holds thanks to Proposition \ref{higherpointwiseestimate}, the sixth one is a consequence of \eqref{Fen4PDV} and the last but one because $\rho(E,F) \geq k$ and $N\geq 1$. 

We turn to the case $p=+\infty$. One has for all $x\in E$,
\[\begin{split}
   |[k\Delta]^jP^{k-1}(f\1_F)(x)| & \lesssim \frac{1}{V(x,k)} \sum_{z\in F} |f(z)| m(z) \left(1+ \frac{\rho(x,z)}{k}\right)^{-N-d} \\
& \quad \leq \|f\|_{L^\infty} \sum_{j\geq 0} \sum_{ 2^j\rho(z,E)\leq \rho(x,z) <2^{j+1}\rho(z,E)} \frac{m(z)}{V(x,k)}\left(1+ \frac{\rho(x,F)}{k}\right)^{-N-d}\\
& \lesssim \|f\|_{L^\infty} \left(1+\frac{\rho(F,E)}{k}\right)^{-N},
  \end{split}\]
where the first line holds because of Proposition \ref{higherpointwiseestimate} and  the last line is obtained as in the case $p=1$. 
\end{proof}

\begin{cor} \label{Fen4GEforResolvant}
The family $(I-(I+s\Delta)^{-1})_{s\in \R^*_+}$ satisfies $L^2$ Gaffney estimates.
\end{cor}

To prove this result, we need the following technical lemma, whose proof can be found on \cite[Lemma B.1]{Fen2}

\begin{lem} \label{expdecay}
For all $m \in \N$, there exists $C_m$ such that for all $s\geq 0$ and $k\in \N$, one has 
 $$\left( \frac{s}{1+s}\right)^k \leq C_m \left(1+ \frac{1+k}{1+s}\right)^{-m}.$$
\end{lem}

\begin{proof} (Corollary \ref{Fen4GEforResolvant})

Since $s\Delta(I+s\Delta)^{-1} = I - (I+s\Delta)^{-1}$, it suffices to show the Davies-Gaffney estimates for $(I+s\Delta)^{-1}$. The  $L^2$-functional calculus provides the identity
\begin{equation} \begin{split}
   (I+s\Delta)^{-1} f&  = \frac{1}{1+s} \left(I-\frac{s}{1+s}P \right)^{-1}f \\
& =  \sum_{k=0}^{+\infty} \frac{1}{1+s}\left(\frac{s}{1+s}\right)^k P^kf,
  \end{split}\end{equation}
where the convergence  holds  in $L^2(\Gamma)$.

Let $N\in \N$, $s\in \N^*$ and $E,F\subset \Gamma$.
Since the operator $(I+s\Delta)^{-1} f$ is uniformly bounded on $s>0$, we can assume without loss of generality that $N\geq 1$ and $\rho(E,F) > s$. Let $f$ be a function supported in $F$. Then, one has with the Gaffney-Davies estimates provided by Proposition \ref{GEforPk}:
\[\begin{split}
   \|(I+s\Delta)^{-1} f\|_{L^2(E)} 
& \lesssim \sum_{k=0}^{+\infty}  \frac{1}{1+s}  \left(\frac{s}{1+s}\right)^k \|P^kf\|_{L^2(E)} \\
& \lesssim \sum_{k=0}^{+\infty}  \frac{1}{1+s} \left(\frac{s}{1+s}\right)^k \left(1+\frac{\rho(E,F)}{1+k} \right)^{-N} \|f\|_{L^2(F)} \\
& \lesssim \|f\|_{L^2(F)} \left[ \sum_{k=0}^{s} \frac{1}{1+s} \left(1+\frac{1+k}{1+s} \right)^{-N}  \left(1+\frac{\rho(E,F)}{1+k} \right)^{-N} \right. \\
& \qquad + \left. \sum_{k=s}^{+\infty} \frac{1+s}{(1+k)^2}  \left(1+\frac{1+k}{1+s} \right)^{-N}  \left(1+\frac{\rho(E,F)}{1+k} \right)^{-N} \right] \\
& \lesssim \|f\|_{L^2(F)} \left(1+\frac{\rho(E,F)}{1+s} \right)^{-N} \left[ \sum_{k=0}^{s} \frac{1}{1+s}  + \sum_{k=s}^{+\infty} \frac{1+s}{(1+k)^2} \right] \\
& \lesssim \|f\|_{L^2(F)} \left(1+\frac{\rho(E,F)}{1+s} \right)^{-N}.\\
  \end{split}\]
 where the third line is a consequence of Lemma \ref{expdecay}.
\end{proof}

The next result proves that the $\epsilon$-molecules (defined later on Section \ref{EqualityHardySpaces}) are uniformly $L^1$ bounded. 

\begin{cor} \label{Fen4BoundedMolecules0}
Let $\epsilon>0$. If $b \in L^2(\Gamma)$ satisfying, for some ball $B=B(x,k)$,
\begin{equation} \label{estimateonb}
\|b\|_{L^2(C_j(x,k)} \leq \frac{2^{-j\epsilon}}{V(2^jB)^\frac12},
\end{equation}
then
$$\|[I-(I+k\Delta)^{-1}] b\|_{L^1} \leq C_\epsilon,$$
where the constant $C_\epsilon$ depends only on $\epsilon$ and $\Gamma$.
\end{cor}

\begin{proof}
Let $x\in \Gamma$ and $k\in \N^*$. Set $B=B(x,k)$ and let $b$ satisfy \eqref{estimateonb}. Define 
$$a = [I-(I+k\Delta)^{-1}]b$$

Corollary \ref{Fen4GEforResolvant} yields that the family $(I-(I+k\Delta)^{-1})_{k}$ satisfies $L^2$ Gaffney estimates. 
Moreover, check that $\rho(C_j(x,k),C_i(x,k)) \gtrsim 2^{\max\{i,j\}} k$  if $|j-i| \geq  B+1$.
Thus, 
\[\begin{split}
   \left\| a  \right\|_1& \leq \sum_{j\geq 0} V(x,2^jk)^{\frac 12} \left\| [I-(I+s\Delta)^{-1}] b \right\|_{L^q(C_j(x,k))} \\
& \lesssim \sum_{i \geq 0} \sum_{j\geq 0} V(x,2^jk)^{\frac 12} \left\| [I-(I+s\Delta)^{-1}] [b\1_{C_i(x,k)}] \right\|_{L^q(C_j(x,k))} \\
& \lesssim  \sum_{|i-j|\geq B+1} V(x,2^jk)^{\frac 12} 2^{-(d+1)j} \left\|b\right\|_{L^2(C_i(x,k))} 
+ \sum_{|i-j|< B+1} V(x,2^jk)^{\frac 12}  \left\|b\right\|_{L^2(C_i(x,k))} \\
& \lesssim \sum_{|i-j|\geq B+1} \left(\frac{V(x,2^jk)}{V(x,2^ik)}\right)^{\frac 12} 2^{-(d+1)j}  2^{-i\epsilon}
+ \sum_{|i-j|< B+1} \left(\frac{V(x,2^jk)}{V(x,2^ik)}\right)^{\frac 12} 2^{-i\epsilon}  \\
& \lesssim \sum_{|i-j|\geq B+1} 2^{jd/2} 2^{-(d+1)j} 2^{-i\epsilon} + \sum_{|i-j|< B+1} 2^{-i\epsilon} \\
& \lesssim 1.
  \end{split}\] 
\end{proof}

\subsection{Gaffney estimates for the gradient} \label{Fen4gradgaffney}

The following result is based on an argument used by Stein in \cite{Stein1} to prove, for $p\in (1,2)$, the $L^p(G)$-boundedness of a vertical Littlewood-Paley functional where $G$ is a Lie group. It has been adapted to the case of graphs by Dungey in \cite{Dungey2} and by the author in \cite{Fen1}.

\begin{prop} \label{Fen4GEforNablaPk}
For any $p\in (1,2)$, the family $(\sqrt{k} \nabla P^{k-1})_{k\in \N^*}$ satisfies $L^p$ Gaffney estimates.
\end{prop}

\begin{proof}
First, assume that $f$ is nonnegative and in $L^1(\Gamma) \cap L^\infty(\Gamma)$. We define for all $k\in \N^*$ and all $p\in (1,2)$ a ``pseudo-gradient'' by
$$ N_p(P^{k-1}f) = -(P^{k-1}f)^{2-p}[\partial_k + \Delta][(P^{k-1}f)^p]$$
where for any function $u_k$ defined on $\Gamma\times \N^*$, $\partial_k u_k = u_{k+1} - u_k$. 

Moreover we define for any function $f$ defined on $\Gamma$ the operator $A$ defined by
$$A f(x) = \sum_{y\sim x} f(y).$$

Propositions 4.6 and 4.7 of \cite{Fen1} state the following results.
\begin{enumerate}[(i)]
 \item For all $x\in \Gamma$, $N_p(P^{k-1}f)(x) \geq 0$. That is, for all $x\in \Gamma$, 
\begin{equation} \label{Fen4Jkisnonnegative} J_kf(x) : = -[\partial_k + \Delta][(P^{k-1}f)^p](x) \geq 0. \end{equation}
 \item For all $p\in (1,2]$, there exists $C=C_p$ such that for all $k\in \N^*$ and all nonnegative function $f \in L^\infty(\Gamma)$, there holds
\begin{equation} \label{Fen4Dungeysresult1} \left|\nabla P^{k-1}f(x)\right|^2 \leq C  \left[A N_p(P^{k-1}f)\right](x).\end{equation}
\end{enumerate}
As a consequence of (ii), if $0 \leq f \in L^\infty$ and $E,F \subset \Gamma$, one has
\[\begin{split}
   \left\|\nabla P^{k-1}f\right\|_{L^p(E)}^p & \lesssim  \left\|A N_p(P^{k-1}f)\right\|_{L^\frac{p}{2}(E)}^{\frac{p}{2}} \\
& \quad \leq \sum_{x\in E}  \left(\sum_{y\sim x}  \left[ N_p(P^{k-1}f)(y) \right]\right)^\frac p2 m(x)  \\
& \quad \lesssim  \sum_{y \in E_{+1}} \left[ N_p(P^{k-1}f)(y) \right]^\frac p2 m(y) \\
& \lesssim  \|[N_p(P^{k-1}f)]^\frac{1}{2}\|_{L^p(E_{+1})}^p,
  \end{split}\]
where
$$E_{+1} = \{y\in \Gamma, \, \exists x\in E: \, x\sim y\} \subset \{y\in \Gamma, \, \rho(y,E) \leq 1\}.$$

It remains to  estimate $\|\left(N_p(P^{k-1}f)\right)^{\frac 12}\|_{L^p(E_{+1})}$. 
Without loss of generality, we can assume that $\rho(E,F) \geq 2^B k \geq 2^B$ (otherwise, \cite[Corollary 1.2]{Dungey2} or \cite[Corollary 4.3]{Fen1} would provide the conclusion of Proposition \ref{Fen4GEforNablaPk}).
Under this assumption, one has $\rho(E_{+1},F) \geq \frac{\rho(E,F)}{2^{B-1}} - 1 \geq \frac{\rho(E,F)}{2^B} \gtrsim \rho(E,F)$.
The proof of the case where $f$ is nonnegative will be thus complete if we prove that for all $p\in (1,2)$, all $N\in \N$ and all $E,F\subset \Gamma$
\begin{equation} \label{Fen4Gaffneyforpseudogradient}
\| \left(N_p(P^{k-1}[f\1_{F}])\right)^{\frac 12}\|_{L^p(E)} \leq \frac{C}{\sqrt k} \left(1+\frac{\rho(E,F)}{k}\right)^{-N} \|f\|_{L^p} \qquad \forall f\in L^p, \, \forall k\in \N^*                                                               
\end{equation}
with some constant $C,c>0$ independent of $E$ and $F$.

In order to do this, we follow the idea of the proof of \cite[Corollary 4.3]{Fen1}. Let $u_k = P^{k-1}[f\1_F]$, then
\begin{equation} \label{Fen4NqPn1} \begin{split}
   \|N_p^{\frac{1}{2}} (u_k)\|_{L^p(E)}^p & = \sum_{x\in E} m(x) N_p^{p/2}(u_k)(x)   \\
& = \sum_{x\in E} m(x) u_k^{\frac{p(2-p)}{2}} J_kf(x)^{p/2} \\
& \leq \left[ \sum_{x\in E} m(x) u_k(x)^p \right]^{\frac{2-p}{2}} \left[ \sum_{x\in E} J_kf(x)m(x) \right]^{\frac{p}{2}} \\
& \leq \|u_k\|_{L^p(E)}^{p(1-\frac p2)} \left[ \sum_{x\in \Gamma} J_kf(x)m(x) \right]^{\frac{p}{2}}
  \end{split} \end{equation}
where the last but one step follows from H\"older inequality and the last one from \eqref{Fen4Jkisnonnegative} stated above. 
Yet,
\[\begin{split}
   \sum_{x\in \Gamma} J_kf(x)m(x) & = - \sum_{x\in \Gamma} \dr_k (u_k^p)(x)m(x) \\
& \leq -p \sum_{x\in \Gamma} m(x) u_k^{p-1}(x) \dr_k u_k(x) \\
& \leq p \|u_k\|_p^{p/p'} \|\dr_k u_k\|_p
  \end{split}\]
where the first line holds because $\sum_{x\in \Gamma}\Delta g(x)m(x)= 0$ if $g\in L^1$, the second line follows from Young inequality, and the third one from H\"older inequality again (with $\frac{1}{p}+\frac{1}{p'}=1$).
Here $\|u_k\|_p \leq \|f\|_p$ while $\|\dr_k u_k \|_p = \|\Delta u_k\|_p  \lesssim \frac{1}{k}\|f\|_p$ by the analyticity of $P$ on $L^p$ (given by \eqref{LB}). Thus
$$\sum_{x\in \Gamma}  J_kf(x)m(x) \lesssim \frac{1}{k} \|f\|_p^p. $$

Substitution of the last estimate in \eqref{Fen4NqPn1} gives
$$\|N_p^{\frac{1}{2}} (u_k)\|_{L^p(E)} \lesssim \frac{1}{\sqrt k}\|f\|_p^{\frac{p}{2}} \|u_k\|_{L^p(E)}^{1-\frac p2},$$
 which ends the proof of \eqref{Fen4Gaffneyforpseudogradient} if we replace $\|u_k\|_{L^p(E)}$ by  the upper estimate given by  Proposition \ref{GEforPk}.

The result for the case where $f\in L^\infty(\Gamma) \cap L^1(\Gamma)$ is deduced by writing $f = f_+ - f_-$, with $f_+ = \max\{0,f\}$ and $f_- = \max\{0,-f\}$.
Then 
\[\begin{split}
   \|\nabla P^{k-1}[f\1_{F}]\|_{L^p(E)} & \leq \|\nabla P^{k-1}[f_+\1_{F}]\|_{L^p(E)} + \|\nabla P^{k-1}[f_-\1_{F}]\|_{L^p(E)} \\
& \leq \frac{C_N}{\sqrt k} \left(1+\frac{\rho(E,F)}{k}\right)^{-N} \left[ \|f_+\|_{p} + \|f_-\|_{p} \right] \\
& \leq \frac{C_N}{\sqrt k} \left(1+\frac{\rho(E,F)}{k}\right)^{-N} \|f\|_{p}.
  \end{split}\]
The result for the general case $f \in L^p(\Gamma)$ is then a consequence of the density of $L^\infty(\Gamma) \cap L^1(\Gamma)$ in $L^p(\Gamma)$.
\end{proof}

\begin{cor} \label{Fen4GEforNablaResolvant}
For any $p\in (1,2)$, the family $(\sqrt s \nabla (I+s\Delta)^{- \frac12})_{s\in \N^*}$ satisfies some weak $L^p$ Gaffney estimates, that is for any $N\in \N$, there exists $C_N$ such that for all $E,F \subset$ and all $s \in \N^*$ satisfying
\begin{equation} \label{scondition}
s \geq \rho(E,F),
\end{equation}
 one has
$$\|\sqrt s \nabla (I+s\Delta)^{- \frac12}[f\1_{F}]\|_{L^p(E)} \leq C_N \left(\frac{\rho(E,F)}{k}\right)^{-N} \|f\|_p.
$$

\end{cor}

\begin{proof}
The proof is similar to the one of Corollary \ref{Fen4GEforResolvant}.

The  $L^p$-functional calculus provides the identity
\begin{equation} \begin{split}
   (I+s\Delta)^{-\frac12} f&  = \frac{1}{\sqrt{1+s}} \left(I-\frac{s}{1+s}P \right)^{-\frac 12}f \\
& =  \sum_{k=0}^{+\infty} \frac{a_k}{\sqrt{1+s}}\left(\frac{s}{1+s}\right)^k P^kf,
  \end{split}\end{equation}
where the convergence holds in $L^p(\Gamma)$ and $\sum a_k z^k$ is the Taylor series of the function $z\mapsto (1-z)^\frac12$. Note that $a_k \simeq \frac1{\sqrt{1+k}}$ (see Lemma B.1 in \cite{Fen1}).

Let $N\in \N$, $s\in \N^*$ and $E,F\subset \Gamma$ satisfying \eqref{scondition}
Let $f$ be a function supported in $F$. Then, one has with the Gaffney-Davies estimates provided by Proposition \ref{GEforPk}:
\[\begin{split}
   \|s \nabla(I+s\Delta)^{-\frac12} f\|_{L^p(E)} 
& \lesssim \sum_{k=0}^{+\infty}  \frac{a_k}{\sqrt{1+s}}  \left(\frac{s}{1+s}\right)^k \|P^kf\|_{L^p(E)} \\
& \lesssim \sum_{k=0}^{+\infty}   \frac{1}{\sqrt{1+k}\sqrt{1+s}} \left(\frac{s}{1+s}\right)^k \left(1+\frac{\rho(E,F)}{1+k} \right)^{-N} \|f\|_{L^p(F)} \\
& \lesssim \|f\|_{L^p(F)} \left[ \sum_{k=0}^{s} \frac{1}{\sqrt{1+k}\sqrt{1+s}} \left(1+\frac{1+k}{1+s} \right)^{-N}  \left(1+\frac{\rho(E,F)}{1+k} \right)^{-N} \right. \\
& \qquad + \left. \sum_{k=s}^{+\infty} \frac{1+s}{(1+k)^2}  \left(1+\frac{1+k}{1+s} \right)^{-N}  \left(1+\frac{\rho(E,F)}{1+k} \right)^{-N} \right] \\
& \lesssim \|f\|_{L^p(F)} \left(1+\frac{\rho(E,F)}{1+s} \right)^{-N} \left[ \sum_{k=0}^{s} \frac{1}{1+s}  + \sum_{k=s}^{+\infty} \frac{1+s}{(1+k)^2} \right] \\
& \lesssim \|f\|_{L^p(F)} \left(1+\frac{\rho(E,F)}{1+s} \right)^{-N}.\\
  \end{split}\]
 where the third line is a consequence of Lemma \ref{expdecay}.
\end{proof}  

As in the previous paragraph, our last result deals with the uniform $L^1$-boundedness of molecules.

\begin{cor} \label{Fen4BoundedMolecules1}
Let $\epsilon>0$. If $b \in L^2(\Gamma)$ satisfying, for some ball $B=B(x,k)$,
\begin{equation} \label{estimateonbb}
\|b\|_{L^2(C_j(x,k)} \leq \frac{2^{-j\epsilon}}{V(2^jB)^\frac12},
\end{equation}
then
$$\|\sqrt{k} \nabla(I+k\Delta)^{-\frac12} b\|_{L^1} \leq C_\epsilon,$$
where the constant $C_\epsilon$ depends only on $\epsilon$ and $\Gamma$.
\end{cor}

\begin{proof}
Let $x\in \Gamma$ and $k\in \N^*$. Set $B=B(x,k)$ and let $b$ satisfy \eqref{estimateonb}. Define 
$$a = \sqrt k \nabla (I+k\Delta)^{-\frac12}b$$

Let $p\in (1,2)$. Check that we have the three following facts:
\begin{itemize}
\item Corollary \ref{Fen4GEforNablaResolvant} yields that the family $(\sqrt k \nabla (I+k\Delta)^{-1})_{k}$ satisfies weak $L^p$ Gaffney estimates. 
\item There holds $\rho(C_j(x,k),C_i(x,k)) \geq 2^{\max\{i,j\}} k$  if $|j-i| \geq  B+1$.
\item The operator $\sqrt k \nabla (I+k\Delta)^{-1}$ is bounded on $L^2$. Indeed, The $L^2$ boundedness of $\nabla \Delta^{-\frac12}$ yields that 
$$\|\sqrt k \nabla (I+k\Delta)^{-1}\|_{2\to 2} \simeq \|(k\Delta)^\frac12(I+k\Delta)^{-\frac12}\|_{2\to 2} = \|[I- (I+k\Delta)^{-1}]^\frac12\|_{2\to 2} \leq 1.$$
\end{itemize}

Moreover, check that
Thus, 
\[\begin{split}
   \left\| a  \right\|_1& \leq \sum_{j\geq 0} V(x,2^jk)^{1-\frac 1p} \left\| [I-(I+s\Delta)^{-1}] b \right\|_{L^p(C_j(x,k))} \\
& \lesssim \sum_{i \geq 0} \sum_{j\geq 0} V(x,2^jk)^{1-\frac 1p} \left\| [I-(I+s\Delta)^{-1}] [b\1_{C_i(x,k)}] \right\|_{L^p(C_j(x,k))} \\
& \lesssim  \sum_{|i-j|\geq B+1} V(x,2^jk)^{1-\frac 1p} \left\| [I-(I+s\Delta)^{-1}] [b\1_{C_i(x,k)}] \right\|_{L^p(C_j(x,k))} \\ 
& \qquad + \sum_{|i-j|< B+1} V(x,2^jk)^{\frac 12}  \left\| [I-(I+s\Delta)^{-1}] [b\1_{C_i(x,k)}] \right\|_{L^2(C_j(x,k))} \\
& \lesssim  \sum_{|i-j|\geq B+1} V(x,2^jk)^{1-\frac1p} 2^{-(d+1)j} \left\|b\right\|_{L^p(C_i(x,k))} 
+ \sum_{|i-j|< B+1} V(x,2^jk)^{\frac 12}  \left\|b\right\|_{L^2(C_i(x,k))} \\
& \lesssim  \sum_{|i-j|\geq B+1} V(x,2^jk)^{1-\frac1p} V(x,2^ik)^{\frac1p-\frac12} 2^{-(d+1)j} \left\|b\right\|_{L^2(C_i(x,k))} 
+ \sum_{|i-j|< B+1} V(x,2^jk)^{\frac 12}  \left\|b\right\|_{L^2(C_i(x,k))} \\
& \lesssim \sum_{|i-j|\geq B+1} \left(\frac{V(x,2^jk)}{V(x,2^ik)}\right)^{1-\frac 1p} 2^{-(d+1)j}  2^{-i\epsilon}
+ \sum_{|i-j|< B+1} \left(\frac{V(x,2^jk)}{V(x,2^ik)}\right)^{\frac 12} 2^{-i\epsilon}  \\
& \lesssim \sum_{|i-j|\geq B+1} 2^{jd(1-1/p)} 2^{-(d+1)j} 2^{-i\epsilon} + \sum_{|i-j|< B+1} 2^{-i\epsilon} \\
& \lesssim 1.
  \end{split}\] 
\end{proof}

\subsection{Off diagonal decay of Lusin functionals}

\begin{prop} \label{Fen4OffDiagonalDecayLbeta}
 Let $\beta>0$. 
Then there exists $C,c>0$ such that, for all $x_0\in \Gamma$ and all sets $E,F \subset \Gamma$ satisfying\begin{enumerate}[(i)]
\item $\sup\{\rho(x_0,y), \, y\in F\} \leq \rho(E,F)$ or
\item $\sup\{\rho(x_0,x), \, x\in E\} \leq \rho(E,F)$,
\end{enumerate} 
there holds for all $s>0$ and all $f\in L^1(\Gamma)\cap L^2(\Gamma)$,
$$\|L_\beta (I-(I+s\Delta)^{-1}) [f\1_F]\|_{L^2(E)} \leq  \frac{C}{V(x_0,\rho(E,F))^\frac12} \frac{s}{\rho(E,F)} \|f\|_{L^1}.$$
\end{prop}

The proof of Proposition \ref{Fen4OffDiagonalDecayLbeta} is similar to the one of \cite[Lemma 3.5]{Fen1} (based itself on Lemma 3.1 in \cite{BRuss}). 

\begin{lem} \label{Fen4l2l1}
Let $r,u\geq 0$, $\alpha<2$. Then
$$\left( \sum_{k\in \N^*} \frac{1}{k} \left[ \frac{k^\alpha}{(k+u)^2} \left(1+\frac{r}{k + u}\right)^{-N} \right]^2 \right)^\frac12 \leq C \sum_{k\in \N^*} \frac{1}{k} \frac{k^\alpha}{(k+u)^2} \left(1+\frac{r}{k + u}\right)^{-N},$$
with a constant $C$ independent on $d$ and $u$.
\end{lem}

\begin{proof}  (Lemma \ref{Fen4l2l1})

If $k\in [2^n,2^{n+1}]$, remark that $k \simeq 2^n$, $k+u \simeq 2^n + u$ and $1+\frac{r}{k + u}\simeq 1+\frac{r}{2^n + u}$ with constants independent on $k$, $n$, $u$ and $r$. Therefore

\[\begin{split}
\left( \sum_{k\in \N^*} \frac{1}{k} \left[ \frac{k^\alpha}{(k+u)^2} \left(1+\frac{r}{k + u}\right)^{-N} \right]^2 \right)^\frac12
& \lesssim \left( \sum_{n\in \N} \left[ \frac{2^{n\alpha}}{(2^n+u)^2}\left(1+\frac{r}{2^n + u}\right)^{-N} \right]^2 \right)^\frac12 \\
& \lesssim  \sum_{n\in \N} \frac{2^{n\alpha}}{(2^n+u)^2} \left(1+\frac{r}{2^n + u}\right)^{-N} \\
& \lesssim  \sum_{k\in \N^*} \frac{1}{k} \frac{k^\alpha}{(k+u)^2} \left(1+\frac{r}{k + u}\right)^{-N}.
\end{split}\]
\end{proof}

\begin{proof} (Proposition \ref{Fen4OffDiagonalDecayLbeta})

Without loss of generality, we can assume that $f$ is supported in $F$. 
We can also assume that $x_0$, $E$ and $F$ satisfy assumption (i) (if they satisfy (ii) instead of (i), the proof is similar). 

First, since $\sum_m \tilde z^m$ is the Taylor series of the function $(1-z)^{-1}$, one has the identity
\begin{equation} \label{Fen4decompinm}
 \begin{split}
    (I-(I+s\Delta)^{-1}) f & = (s\Delta) (I+s\Delta)^{-1}f \\
& = \frac{s\Delta}{1+s} \left(I - \frac{s}{1+s} P\right)^{-1}f \\
& = \frac{s\Delta}{1+s} \sum_{m=0}^{+\infty} \left(\frac{s}{1+s}\right)^m P^m f \\
& = s\Delta \sum_{m=0}^\infty b_m P^mf,
 \end{split}
\end{equation}
where $b_m : = \frac{s^m}{(1+s)^{m+1}}$ and the series converges in $L^2(\Gamma)$. Notice that 
\begin{equation} \label{Fen4sumbm}
\sum_{m=0}^\infty b_m = 1.
\end{equation}
Moreover, let $\kappa$ be the only integer such that $\kappa < \beta + 1 \leq \kappa + 1$. 
Since $\beta>0$ and $\kappa$ is an integer, notice that 
\begin{equation}\label{Fen4Mleqkappa} 1\leq \kappa \end{equation}
If $\sum_l a_l z^l$ is the Taylor series of the function $(1-z)^{\beta-\kappa}$, then one has
\begin{equation} \label{Fen4decompinl}
 \Delta^{\beta+1} f  = \sum_{l\geq 0} a_l P^l \Delta^{\kappa+1} f
\end{equation}
where the sum converges in $L^2(\Gamma)$ (see \cite[Proposition 2.1]{Fen2} for the proof of the convergence).

\noindent The Minkowski inequality together with the identities \eqref{Fen4decompinm} and \eqref{Fen4decompinl} yields
\[\begin{split}
\|L_\beta & (I-(I+s\Delta)^{-1}) f\|_{L^2(E)}  \\
& = \left( \sum_{k\geq 1} k^{2\beta-1} \sum_{x\in E} \frac{m(x)}{V(x,k)} \sum_{y\in B(x,k)} m(y) |\Delta^{\beta} P^{k-1} (I-(I+s\Delta)^{-1}) f(y)|^2\right)^\frac12 \\
& \leq \left( \sum_{k\geq 1} k^{2\beta-1}  \sum_{y\in D_k(E)} m(y) |\Delta^{\beta} P^{k-1} (I-(I+s\Delta)^{-1}) f(y)|^2 \sum_{x\in B(y,k)} \frac{m(x)}{V(x,k)}\right)^\frac12 \\
& \lesssim \left( \sum_{k\geq 1} k^{2\beta-1}  \|\Delta^{\beta} P^{k-1} (I-(I+s\Delta)^{-1}) f\|_{L^2(D_k(E))}^2 \right)^\frac12 \\
& \quad \leq  s \sum_{l,m\geq 0} a_lb_m \left( \sum_{k\geq 1} k^{2\beta-1} \|\Delta^{1+\kappa} P^{k+l+m-1} f\|_{L^2(D_k(E))}^2\right)^\frac12 \\
&  \quad : = s \sum_{l,m\geq 0} a_lb_m \left( \sum_{k\geq 1}\frac{1}{k} I(k,l,m)^2 \right)^\frac12,
\end{split}\]
where $D_k(E)$ denotes the set $\{y\in \Gamma, \, \rho(y,E) < k\}$

\noindent We want to get the following estimate estimate: there exists $c>0$ such that
\begin{equation} \label{Fen4Iklmestimate}
 I(k,l,m) \lesssim J(k,l,m): = k^{1+\beta-\kappa} \frac{1}{V(x_0,\rho(E,F))} \frac{1}{(k+l+m)^{2}} \left(1+ \frac{\rho(E,F)}{k+l+m} \right)^{-N}  \|f\|_{L^1}.
\end{equation}

\noindent We will first establish \eqref{Fen4Iklmestimate} when $k \leq \frac{\rho(E,F)}{2^B}$. In this case, notice that
$$\rho(D_k(E),F) \geq \frac{\rho(E,F)}{2^{B-1}} - k \geq \frac{\rho(E,F)}{2^B}$$
and thus $(D_k(E),F,x_0)$ satisfies the assumption of Corollary \ref{Fen4LqLpforDeltaPk}, which implies
\[\begin{split}
   I(k,l,m) & \lesssim \frac{k^\beta}{(k+l+m)^{\kappa + 1}} \frac{1}{V(x_0,k+l+m)^{\frac12}} \left(1+ \frac{\rho(E,F)}{2^B(k+l+m)} \right)^{-N-\frac d2} \|f\|_{L^1} \\
&  \lesssim \frac{k^{\beta+1-\kappa}}{(k+l+m)^{2}} \frac{1}{V(x_0,\rho(E,F))^{\frac12}} \left(1+ \frac{\rho(E,F)}{k+l+m} \right)^{-N}  \|f\|_{L^1}
  \end{split}\]
where the last line holds thanks to estimate \eqref{Fen4Mleqkappa} and Proposition \ref{Fen4propDV}.

\noindent Otherwise, $k > \frac{\rho(E,F)}{2^B}$ and then $(F,x_0,k+l+m)$ satisfies the assumption of Corollary \ref{Fen4LqLpforDeltaPk}, which yields
\begin{equation}\begin{split}
   I(k,l,m) & \lesssim \frac{k^\beta}{(k+l+m)^{\kappa + 1}} \frac{1}{V(x_0,k+l+m)^{\frac12}} \|f\|_{L^1} \\
&  \lesssim \frac{k^{\beta+1-\kappa}}{(k+l+m)^{2}} \frac{1}{V(x_0,\rho(E,F))^{\frac12}} \|f\|_{L^1} \\
&  \lesssim \frac{k^{\beta+1-\kappa}}{(k+l+m)^{2}} \frac{1}{V(x_0,\rho(E,F))^{\frac12}} \left(1+ \frac{\rho(E,F)}{k+l+m} \right)^{-N}  \|f\|_{L^1}
  \end{split}\end{equation}
where the second line holds thanks to \eqref{DV} and \eqref{Fen4Mleqkappa}. This ends the proof of \eqref{Fen4Iklmestimate}.

\noindent Recall that $1+\beta-\kappa \in (0,1]$. Then Lemma \ref{Fen4l2l1} implies
\begin{equation}\label{Fen4calculusIklm1}\begin{split}
  \|L_\beta & (I-(I+s\Delta)^{-1})  f\|_{L^2(E)} \lesssim s \sum_{k,l,m\geq 0} \frac{1}{k} a_lb_m  J(k,l,m) \\
& \lesssim s^M \frac{\|f\|_{L^1}}{V(x_0,\rho(E,F))^\frac12} \sum_{m\geq 0} b_m \sum_{n\geq 1} \frac{1}{(n+m)^{2}} \left(1+ \frac{\rho(E,F)}{n+m} \right)^{-N}  \sum_{l=0}^{n-1} a_l (n-l)^{\beta-\kappa}.
\end{split}\end{equation}
We claim
$$\sum_{l=0}^{n-1} a_l (n-l)^{M+\beta-\kappa-1} \lesssim 1.$$
Indeed, when $\beta=\kappa$, one has $a_l = 1$ if $l=0$ and equals $0$ otherwise. Therefore the estimate is true. 
If $\beta<\kappa$, \cite[Lemma B.1]{Fen1} yields that $a_l \lesssim (l+1)^{\kappa-1-\beta}$ and therefore
$$\sum_{l=0}^{n-1} a_l (n-l)^{\beta-\kappa} \lesssim \int_0^1 t^{\kappa-1-\beta}(1-t)^{\beta-\kappa} dt < +\infty.$$

Besides, 
\[\begin{split}
   \sum_{n\geq 1} \frac{1}{(n+m)^{2}} \left(1+ \frac{\rho(E,F)}{n+m} \right)^{-N}
& = \rho(E,F)^{-2} \sum_{n\geq 1} \left(\frac{\rho(E,F)}{n+m}\right)^{2} \left(1+ \frac{\rho(E,F)}{n+m} \right)^{-N}  \\
& \lesssim \rho(E,F)^{-2} \left[ \sum_{n=1}^{\rho(E,F)} 1 + \sum_{n > \rho(E,F)} \left(\frac{\rho(E,F)}{m+n}\right)^{2} \right] \\
& \lesssim \rho(E,F)^{-1}.
  \end{split}\]
Consequently, the estimate \eqref{Fen4calculusIklm1} yields
\[\begin{split}
  \|L_\beta (I-(I+s\Delta)^{-1}) f\|_{L^2(E)} & \lesssim \frac{1}{V(x_0,\rho(E,F))^\frac12} \frac{s}{\rho(E,F)}  \|f\|_{L^1} \sum_{m\geq 0} b_m \\
& \quad = \frac{1}{V(x_0,\rho(E,F))^\frac12} \frac{s}{\rho(E,F)}  \|f\|_{L^1}
  \end{split}\]
where the last line is due to \eqref{Fen4sumbm}.
\end{proof}

In the same way, we have

\begin{prop} \label{Fen4OffDiagonalDecayLbeta2}
 With the same assumptions as Proposition \ref{Fen4OffDiagonalDecayLbeta}, for any  $f\in L^1(\Gamma) \cap L^2(\Gamma)$, there holds
$$\|L_\frac12 (I-(I+s\Delta)^{-1})^\frac12 [f\1_F]\|_{L^2(E)} \leq  \frac{C}{V(x_0,\rho(E,F))^\frac12} \left(\frac{s}{\rho(E,F)}\right)^\frac12 \|f\|_{L^1}.$$
\end{prop}

Let us define a discrete version of the Littlewood-Paley functionals, that can be found in \cite{Fen1}. 
For any $\beta>0$, the functional $g_\beta$ is defined as
$$g_\beta f(x) = \left( \sum_{k=1}^\infty k^{2\beta-1} |\Delta^\beta P^{k-1}f(x)|^2 \right)^2.$$

\begin{prop} \label{Fen4OffDiagonalDecayLbetagbeta} 
 With the same assumptions as Proposition \ref{Fen4OffDiagonalDecayLbeta}, for any  $f\in L^1(\Gamma) \cap L^2(\Gamma)$, there holds
$$\|L_\beta (I-P^k) [f\1_F]\|_{L^2(E)} \leq  \frac{C}{V(x_0,\rho(E,F))^\frac12} \frac{k}{\rho(E,F)} \|f\|_{L^1}$$
and
$$\|g_\beta (I-P^k) [f\1_F]\|_{L^2(E)} \leq  \frac{C}{V(x_0,\rho(E,F))^\frac12} \frac{k}{\rho(E,F)} \|f\|_{L^1}$$
\end{prop}

\begin{prop} \label{Fen4OffDiagonalDecaygbeta2}
Let $m>0$ a real number. If we have the same assumptions as Proposition \ref{Fen4OffDiagonalDecayLbeta}, then for any  $f\in L^1(\Gamma) \cap L^2(\Gamma)$, there holds
$$\|g_\beta (k\Delta)^m [f\1_F]\|_{L^2(E)} \leq \frac{C}{V(x_0,\rho(E,F))^\frac12} \left(\frac{k}{\rho(E,F)}\right)^m \|f\|_{L^1}$$
\end{prop}

\begin{proof}
The proofs of these three propositions are similar to the one of Proposition \ref{Fen4OffDiagonalDecayLbeta} and are left to the reader. 
See also \cite[Lemma 3.5]{Fen1} and \cite[Lemmata 2.14 and 2.18]{Fen2}.
\end{proof}

As a consequence, we have the following result

\begin{theo} \label{Fen4LpboundednessofLbeta} 
Then for all $\beta>0$, the functional $L_\beta$ is bounded on $L^p(\Gamma)$ for any $p\in (1,2]$ and also bounded from $L^{1,\infty}(\Gamma)$ to $L^1(\Gamma)$.

Moreover, if $g_\beta$ is the discrete Littlewood-Paley quadratic functional defined for any $\beta>0$ as 
$$g_\beta f(x) = \left( \sum_{k=1}^\infty k^{2\beta-1} |\Delta P^{k-1}f(x)|^2 \right)^2,$$
then $g_\beta$ is also bounded on $L^p(\Gamma)$ for any $p\in (1,2]$ and bounded from $L^{1,\infty}(\Gamma)$ to $L^1(\Gamma)$.
\end{theo}

\begin{proof}
We set $A_B = P^{k_B}$. 
It is then a straightforward consequence of Theorem \ref{Fen4BKtheoremforinterpolation}, Proposition \ref{Fen4OffDiagonalDecayLbetagbeta} and Corollary \ref{Fen4LqLpforDeltaPk}.
\end{proof}

\section{Equality of Hardy spaces}

\label{EqualityHardySpaces}

In this section, $(\Gamma, \mu)$ is a weighted graph as defined in Section \ref{defgraphs} verifying \eqref{LB}. We assume that there exists a function $\beta$ bounded from below by 1 and from above by $B$ such that $\Gamma$ satisfies \eqref{DV} and \eqref{UE}. Under this circumstances, $\rho$ denotes the quasidistance $\d^\beta$ and the metric considered is the one of $\rho$ (see Section \ref{metric}).

\subsection{Definition of Hardy spaces}

\noindent We  define two kinds  of Hardy spaces. The first one is defined using molecules.

\begin{defi} \label{Fen4molec0}
Let $\epsilon \in (0,+\infty)$. A function $a\in L^2(\Gamma)$ is called an $\epsilon$-molecule (on $\Gamma$) if there exist $x\in \Gamma$, $k\in \N^*$ and a function $b \in  L^2(\Gamma)$ such that
\begin{enumerate}[(i)]
 \item $ a = [I-(I+k\Delta)^{-1}]b$, 
 \item $\ds \|b\|_{L^2(C_j(x,k))} \leq 2^{-j\epsilon} V(x,2^jk)^{-\frac{1}{2}}$, for all $j\geq 0$.
\end{enumerate}
We say that  an  $\epsilon$-molecule $a$ is associated with a vertex $x$ and an integer $k$ when we want to refer to $x$ and $k$ given by the definition.
\end{defi}

\begin{defi} \label{Fen4molec1}
 Let $\epsilon \in (0,+\infty)$. 
A function $a \in L^2(T_\Gamma)$ is called  an  $\epsilon$-molecule (on $T_\Gamma$) if there exist $x\in \Gamma$, $k\in \N^*$ and a function $b \in L^2(\Gamma)$ such that
\begin{enumerate}[(i)]
 \item $a = \sqrt k d (I+k\Delta)^{-\frac{1}{2}} b$;
 \item $\|b\|_{L^2(C_j(x,k))} \leq 2^{-j\epsilon} V(x,2^jk)^{-\frac{1}{2}}$ for all $j\geq 0$.
\end{enumerate}
\end{defi}

Corollaries \ref{Fen4BoundedMolecules0} and \ref{Fen4BoundedMolecules1} state

\begin{prop} \label{Fen4sizemolec}
As it has been seen in Propositions \ref{Fen4BoundedMolecules0} and \ref{Fen4BoundedMolecules1}, when $a$ is a molecule occurring in Definition \ref{Fen4molec0} or in Definition \ref{Fen4molec1}, one has 
$$\|a\|_{L^1} \lesssim 1.$$
\end{prop}

\begin{defi}
Let $\epsilon \in (0,+\infty)$. We say that a function $f$ on $\Gamma$ (resp. on $T_\Gamma$) belongs to $H^1_{mol,\epsilon}(\Gamma)$ (resp. $H^1_{mol,\epsilon}(T_\Gamma)$) if $f$ admits an $\epsilon$-molecular representation, 
that is if there exist a sequence $(\lambda_i)_{i\in \N} \in \ell^1$ and a sequence $(a_i)_{i\in \N}$ of $\epsilon$-molecules on $\Gamma$ (resp. on $T_\Gamma$) such that
\begin{equation} \label{Fen4sumf}
f=\sum_{i=0}^\infty \lambda_i a_i
\end{equation}
where the convergence of the  series  to $f$ holds pointwise.  Define, for all $f\in H^1_{mol,\epsilon}(\Gamma)$ (resp. $f \in H^1_{mol,\epsilon}(T_\Gamma)$),  $$\|f\|_{H^1_{mol,\epsilon}} = \inf\left\{ \sum_{j=0}^\infty |\lambda_j|, \ \sum_{j=0}^\infty \lambda_j a_j, \text{ is an $\epsilon$-molecular representation of $f$} \right\}.$$
\end{defi}


\begin{prop} \label{Fen4H1complete} Let $\epsilon \in (0,+\infty)$. Let $E$ be the graph $\Gamma$ or the tangent bundle $T_\Gamma$.
 Then:
\begin{enumerate}[(i)]
\item
the map $f\mapsto \left\Vert f\right\Vert_{H^1_{mol,\epsilon}}$ is a norm on the space $H^1_{mol,\epsilon}(E)$,
\item
the space $H^1_{mol,\epsilon}(E)$ is complete and continuously embedded in $L^1(E)$. \end{enumerate}
\end{prop}

\begin{proof}
Proposition \ref{Fen4sizemolec} shows that, if $f$ is in $H^1_{mol,\epsilon}(E)$ the series \eqref{Fen4sumf} converges in $L^1(E)$, and therefore converges to $f$ in $L^1(E)$.  This yields $(i)$ and the embeddings in $(ii)$. 
Moreover, a normed linear vector space $X$ is complete if and only if it has the property
$$\sum_{j=0}^\infty \|f_j\|_{X} < +\infty \Longrightarrow \sum_{j=0}^\infty f_j  \mbox{ converges in }  X.$$
Using this criterion, the completeness of the Hardy spaces under consideration is a straightforward consequence of the fact that $\|a\|_{L^1} \lesssim 1$ whenever $a$ is a molecule.  See also the argument for the completeness of $H^1_L$ in \cite{HoMay}, p. 48.
\end{proof}

The second kind  of Hardy spaces  is defined via quadratic functionals. 

\begin{defi}
 Define, for $\beta>0$, the quadratic functional $L_\beta$ on $L^2(\Gamma)$ by
$$L_\beta f(x) =  \left( \sum_{(y,k) \in \gamma(x)} \frac{k^{2\beta-1}}{V(x,k)} |\Delta^{\beta} P^{k-1} f(y)|^2m(y) \right)^{\frac{1}{2}}$$
where $\gamma(x) = \left\{ (y,k) \in \Gamma \times \N^*, \, \rho(x,y) < k \right\}$.
\end{defi}

\begin{defi}
The space $E^1_{quad,\beta}(\Gamma)$ is defined for all $\beta >0$ by
$$E^1_{quad,\beta}(\Gamma) : = \left\{ f\in L^2(\Gamma), \, \|L_\beta f\|_{L^1} < +\infty\right\}.$$
It is outfitted with the norm
$$\|f\|_{H^1_{quad,\beta}} : = \|L_\beta f\|_{L^1}.$$
The space $E^1_{quad,\beta}(T_\Gamma)$ is defined from $E^1_{quad,\beta}$ as
$$E^1_{quad,\beta}(T_\Gamma) : = \left\{ f\in H^2(T_\Gamma), \,  \Delta^{-\frac12} d^* f \in E^1_{quad,\beta}(\Gamma) \right\}.$$
It is outfitted with the norm
$$\|f\|_{H^1_{quad,\beta}} : = \|L_\beta \Delta^{-\frac12} d^* f\|_{L^1}.$$
\end{defi}

\begin{rmk}
The fact that the map $f\mapsto \|f\|_{H^1_{quad,\beta}}$ is a norm is proven in \cite[Remark 1.20]{Fen2}. 
\end{rmk}

\subsection{$H^1_{mol} \cap L^2 \subset E^1_{quad}$}

\begin{prop} \label{Fen4Main11}
Let $\epsilon>0$ and $\beta >0$. Then $H^1_{mol,\epsilon}(\Gamma) \cap L^2(\Gamma) \subset E^1_{quad,\beta}(\Gamma)$ and
$$\|f\|_{H^1_{quad,\beta}} \lesssim \|f\|_{H^1_{mol,\epsilon}} \qquad \forall f\in H^1_{mol,\epsilon}(\Gamma) \cap L^2(\Gamma).$$
\end{prop}

\begin{proof}
 The proof follows the idea of the one in \cite[Proposition 4.1]{Fen2}. Let $f\in H^1_{mol,\epsilon} \cap L^2(\Gamma)$. 
Then there exist $(\lambda_i)_{i\in \N} \in \ell^1$ and a sequence $(a_i)_{i\in \N}$ of $\epsilon$-molecules such that $f = \sum \lambda_i a_i$ where the convergence is in $L^1(\Gamma)$ and 
$$\sum_{i\in \N} |\lambda_i| \simeq \|f\|_{H^1_{mol,M,p,\epsilon}}.$$

First, since  $\|P^k\|_{1\to 1}\leq 1$  for all $k\in \N$, the operators $\Delta^\beta$ and then $\Delta^\beta P^{k-1}$ are $L^1$-bounded for $\beta >0$ (see \cite{CSC}). Consequently,
$$\Delta^\beta P^{l-1} \sum_{i\in \N} \lambda_i a_i = \sum_{i\in \N} \lambda_i \Delta^\beta P^{l-1} a_i$$
with convergence in $L^1(\Gamma)$.
Since the $L^1$-convergence implies the pointwise convergence, it yields, for all $x\in \Gamma$,
\[\begin{split}
   \left|\Delta^\beta P^{k-1} \sum_{i\in \N} \lambda_i a_i(x)\right| & = \left|\sum_{i\in \N} \lambda_i \Delta^\beta P^{k-1} a_i(x) \right| \\
& \leq \sum_{i\in \N} |\lambda_i| \left|\Delta^\beta P^{k-1} a_i(x) \right|.
  \end{split}\]
From here, the estimate
$$\|L_\beta f\|_{L^1}  = \left\| L_\beta \sum_{i\in \N} \lambda_i a_i \right\|_{L^1} \lesssim \sum_{i\in \N} |\lambda_i| \|L_\beta a_i\|_{L^1}$$
is  just  a consequence of the generalized Minkowski inequality.

It remains to prove that there exists a constant $C$ such that for all $\epsilon$-molecules $a$, one has
\begin{equation} \label{Fen4BoundenessOnAtoms}
 \|L_\beta a\|_{L^1} \leq C.
\end{equation}

Let $x\in \Gamma$ and $k\in \N^*$ associated with the $\epsilon$-molecule $a$. By H\"older inequality and the doubling property, we can write
\begin{equation} 
 \|L_\beta a\|_{L^1} \lesssim  \sum_{j=0}^\infty V(x,2^js)^{\frac{1}{2}} \|L_\beta a\|_{L^2(C_j(x,k))}.
\end{equation}
Besides, we write $a =  (I-(I+k\Delta)^{-1})b$ and then $\ds b= \sum_{i\geq 0} [\1_{C_i(x,k)} b]$ to get
\begin{equation}
\begin{split}
 \|L_\beta a\|_{L^1} & \lesssim  \sum_{i,j\geq 0} V(x,2^js)^{\frac{1}{2}} \|L_\beta (I-(I+k\Delta)^{-1}) [\1_{C_i(x,k)} b] \|_{L^2(C_j(x,k))} \\
 & \lesssim \sum_{|i-j| \leq B} V(x,2^jk)^{\frac{1}{2}} \|L_\beta (I-(I+k\Delta)^{-1}) [\1_{C_i(x,k)} b] \|_{L^2(C_j(x,k))} \\
& \qquad + \sum_{|i-j|\geq B+1} V(x,2^jk)^{\frac{1}{2}} \|L_\beta (I-(I+k\Delta)^{-1}) [\1_{C_i(x,k)} b] \|_{L^2(C_j(x,k))} \\
& : = I_1 + I_2
 \end{split}
\end{equation}
The term $I_1$ is evalutated first. The $L^2$-boundedness of $L_\beta$ (see Theorem \ref{Fen4LpboundednessofLbeta}) and the uniform $L^2$-boundedness of $I-(I+s\Delta)^{-1}$ implies 
$$\left\| L_\beta (I-(I+s\Delta)^{-1}) [\1_{C_i(x,k)} b] \right\|_{L^2(C_j(x,k))} \lesssim \|b\|_{L^2(C_i(x,k))} \leq \frac{2^{-i\epsilon}}{V(x,2^{i}k)^\frac12}.$$
And with the doubling property \eqref{DV}, 
\begin{equation}\label{I2estimate}
I_1 \lesssim \sum_{|i-j| \leq B} 2^{-i\epsilon} \left(\dfrac{V(x,2^jk)}{V(x,2^ik)}\right)^{\frac{1}{2}} \lesssim \sum_{i\geq 0} 2^{-i\epsilon}  \lesssim 1.
\end{equation} 
We turn to the estimate of $I_2$. If $|i-j| \geq B+1$, one has
$$\rho(C_j(x,k),C_i(x,k)) \geq 2^{\max\{i,j\}}k \geq 2^{\min\{i,j\}+B+1} k \geq \min\left\{ \sup_{y\in C_j(x,k)} \rho(x,y), \sup_{y\in C_i(x,k)} \rho(x,y) \right\}.$$
That is $C_j(x,k)$, $C_i(x,k)$ and $x$ satisfy the assumption of Proposition \ref{Fen4OffDiagonalDecayLbeta}. Consequently, 
\[\begin{split}
 I_2 & \lesssim \sum_{|i-j|\geq B+1} V(x,2^jk)^\frac12 2^{-j(d+1)} \left\|b\right\|_{L^2(C_i(x,k))} \\
& \lesssim  \sum_{|i-j|\geq B+1} 2^{-j-i\epsilon} 2^{-jd} \left(\dfrac{V(x,2^jk)}{V(x,2^jk)}\right)^{\frac{1}{2}} \\
& \lesssim \sum_{|i-j|\geq B+1}  2^{-j-i\epsilon} \\
& \lesssim  1.
  \end{split}\]
where the third line is a consequence of Proposition \ref{Fen4propDV}.
We conclude by noticing that the estimates on $I_1$ and $I_2$ provides exactly \eqref{Fen4BoundenessOnAtoms}.
\end{proof}

\begin{prop} \label{Fen4Main21}
Let $\epsilon>0$. Then $H^1_{mol,\epsilon}(T_\Gamma) \cap H^2(T_\Gamma) \subset E^1_{quad,\frac12}(T_\Gamma)$ and
$$\|f\|_{H^1_{quad,\frac12}} \lesssim \|f\|_{H_{mol,\epsilon}^1} \qquad \forall f\in H^1_{mol,\epsilon}(T_\Gamma) \cap H^2(T_\Gamma).$$
\end{prop}

\begin{proof}
The proof is similar to the previous one then we will point out only the main differences. 
Since $d^*$ and $P^l$ are $L^1$-bounded, it is enough to prove the uniform boundedness of $\|L_\frac{1}{2} \Delta^{-\frac12} d^* a\|_{L^1}$ when $a$ is an $\epsilon$-molecule.

However, notice that if $a=\sqrt k d (I+k\Delta)^{-\frac12} b$ (with $k$, $b$ associated with $a$), one has
$$\|L_\frac{1}{2} \Delta^{-\frac12} d^* a\|_{L^1} = \|L_\frac{1}{2} (I-(I+k\Delta)^{-1})^{\frac12} b\|_{L^1}.$$
We conclude then as in Proposition \ref{Fen4Main11}, using Proposition \ref{Fen4OffDiagonalDecayLbeta2} instead of Proposition \ref{Fen4OffDiagonalDecayLbeta}.
\end{proof}

\begin{prop} \label{Fen4Main32}
The space $E^1_{0}(\Gamma)$ is continuously embedded in $E^1_{quad,\frac12}(\Gamma)$ and
$$\|f\|_{H^1_{quad,\frac12}} \lesssim \|f\|_{E^1_0} \qquad \forall f\in E^1_0(\Gamma).$$
\end{prop}

\begin{proof}
We refer to subsection \ref{Fen4InterpolationSubsection} for the definition of $E^1_0(\Gamma)$ and of atoms.

Due to the definition of $E^1_0(\Gamma)$, 
we only need to check that the quantity $\|L_\beta a\|_{L^1}$ is uniformly bounded on atoms.
The proof is analogous to the one of Proposition \ref{Fen4Main11}, using Proposition \ref{Fen4OffDiagonalDecayLbetagbeta}  instead of Proposition \ref{Fen4OffDiagonalDecayLbeta}.
\end{proof}

\subsection{$E^1_{quad} \subset H^1_{mol} \cap L^2$}

Let us introduce first the functional $\pi_{\eta,\beta}:  \ T^2(\Gamma)\to L^2(\Gamma)$ defined for any real $\beta>0$ and any integer $\eta > \beta$ by
$$\pi_{\eta,\beta} F(x) = \sum_{l\geq 1} \frac{c_l^\eta}{l^\beta} \left[\Delta^{\eta-\beta}(I+P)^\eta P^{l-1} F(.,l)\right](x)$$
where $\ds \sum_{l\geq 1} c_{l}^\eta z^{l-1}$ is the Taylor series of the function $(1-z)^{-\eta}$. Let us recall Lemma 4.8 in \cite{Fen2}:

\begin{lem}
 The operator $\pi_{\eta,\beta}$ is bounded from $T^2(\Gamma)$ to $L^2(\Gamma)$.
\end{lem}


\begin{lem} \label{molecules}
 Suppose that $A$ is a $T^1(\Gamma)$-atom associated with a ball $B \subset \Gamma$. 
Then for every $M\in \N^*$, $\beta>0$ and $\epsilon\in (0,+\infty)$, there exist an integer $\eta = \eta_{\beta,\epsilon}$ and a uniform constant $C_{\beta,\epsilon}>0$ such that 
$C_{\beta,\epsilon}^{-1} \pi_{\eta,\beta}(A)$ is a $\epsilon$-molecule associated with the ball $B$.
\end{lem}

\begin{proof}
Let $d_0$ be the value that appears in Proposition \ref{Fen4propDV}. Set $\eta = \lceil \frac{d_0}{2} + \epsilon+ \beta \rceil + 2$, that is the only integer such that 
$$\eta \geq \frac{d_0}2 + \epsilon + \beta + 2 > \eta -1.$$

Let $A$ be a $T^1$-atom associated with a ball $B$ of radius $k$ and center $x_0$. We write
$$a : = \pi_{\eta,\beta}(A) = k\Delta(I+k\Delta)^{-1} b$$
where
$$b:= \sum_{l\geq 1} \frac{c_l^\eta}{l^\beta} \left(\frac{I+k\Delta}{k}\right) \Delta^{\eta-\beta-1}(I+P)^\eta P^{l-1} A(.,l)$$
Let us check that $a$ is an $\epsilon$-molecule associated with $B$, up to multiplication by some harmless constant $C_{\beta,\epsilon}$. First, one has, for all $h\in L^2(B(x_0,2\eta^B k))$,
\[\begin{split}
   \left| \left< b,h \right> \right|& \leq  \sum_{l\geq 1} \frac{c_l^\eta}{l^\beta} \left|\left<(k^{-1}+ \Delta)\Delta^{\eta-\beta-1}(I+P)^\eta P^{l-1} A(.,l),h \right> \right| \\
& = \sum_{l\geq 1} \frac{c_l^\eta}{l^\beta} \left|\left< A(.,l),(k^{-1}+ \Delta)\Delta^{\eta-\beta-1}(I+P)^\eta P^{l-1} h \right>\right| \\
& \lesssim \sum_{l\geq 1} l^{\eta-\beta-1} \|A(.,l)\|_{L^2(B)}\|(k^{-1}+ \Delta)\Delta^{\eta-\beta-1}(I+P)^\eta P^{l-1}h\|_{L^2(B)}\\ 
& \lesssim  \|A\|_{T^2} \left(  \sum_{l= 1}^{k} l^{2(\eta-\beta)-1} \|(k^{-1}+ \Delta)\Delta^{\eta-\beta-1}(I+P)^\eta P^{l-1}h\|_{L^2(B)}^2  \right)^{\frac12} \\
& \lesssim \|A\|_{T^2} \left(  \sum_{l= 1}^{k} l^{2(\eta-\beta)-1} \left[ \left\|\frac1l \Delta^{\eta-\beta-1}(I+P)^\eta P^{l-1}h\right\|_{L^2(B)}^2 
+ \|\Delta^{\eta-\beta}(I+P)^\eta P^{l-1}h\|_{L^2(B)}^2 \right]\right)^{\frac12} \\
& \lesssim  \|A\|_{T^2} \left( \left\| g_{\eta-\beta-1}(I+P)^\eta h \right\|_{L^2} + \left\| g_{\eta-\beta}(I+P)^\eta h \right\|_{L^2} \right) \\
& \lesssim  \|A\|_{T^2} \left\| (I+P)^\eta h \right\|_{L^2} \\
& \lesssim \frac{1}{V(B)^{\frac{1}{2}}} \left\|h \right\|_{L^2}
\end{split}\]
where we used the $L^2$-boundedness of the quadratic Littlewood-Paley functional for the last but one line (see \cite[Lemma 3.2]{Fen1}, \cite{LMX}).

Let $j> B\log_2(\eta)>1$ and $h\in L^2(C_j(x_0,k))$. Note that $y \in \Supp (I+P)^\eta h$ implies that $\rho(y,x_0) \geq 2^{j+B-1} k$. Thus $\rho(\Supp (I+P)^\eta h, B) \gtrsim 2^jk$. Check then that
\[\begin{split}
   \left| \left< b,h \right> \right|&  \lesssim \|A\|_{T^2} \frac1k\left(  \sum_{l= 1}^{k} l^{2(\eta-\beta)-1} \|(1+k\Delta)\Delta^{\eta-\beta-1}(I+P)^\eta P^{l-1}h\|_{L^2(B)}^2  \right)^{\frac12} \\
& \lesssim k^{\eta-\beta-2} \|A\|_{T^2} \left(  \sum_{l= 1}^{k} \left[ l \|\Delta^{\eta-\beta-1}(I+P)^\eta P^{l-1}h\|_{L^2(B)}^2  
+ l^3 \|\Delta^{\eta-\beta}(I+P)^\eta P^{l-1}h\|_{L^2(B)}^2 \right]\right)^{\frac12} \\
& \lesssim k^{\eta-\beta-2}   \|A\|_{T^2} \left[ \left\| G_{1} \Delta^{\eta-\beta-2} (I+P)^\eta h \right\|_{L^2(B)} 
+ \left\| G_{1} \Delta^{\eta-\beta-2} (I+P)^\eta h \right\|_{L^2(B)}\right] \\
& \lesssim \frac{1}{V(x_0,2^jk)}\frac{k^{\eta-\beta-2}}{(2^jk)^{\eta-\beta-2}} \|A\|_{T^2} \left\| (I+P)^\eta h \right\|_{L^1} \\
& \lesssim 2^{-j(\frac{d_0}{2} + \epsilon)} \|A\|_{T^2} \left\|h \right\|_{L^2} \\
& \lesssim \frac{2^{-j(\frac{d_0}{2} + \epsilon)}}{V(B)^{\frac{1}{2}}} \left\|h \right\|_{L^2}\\
& \lesssim \frac{2^{-j\epsilon}}{V(2^jB)^{\frac{1}{2}}} \left\|h \right\|_{L^2}\\
\end{split}\]
where we used Proposition \ref{Fen4OffDiagonalDecaygbeta2} for the forth line and Proposition \ref{Fen4propDV} for the last one.
We conclude that, up to multiplication by some harmless constant, $b$ is an $\epsilon$-molecule.
\end{proof}

\begin{prop} \label{Fen4Main12}
Let $\epsilon>0$ and $\beta>0$. Then $E^1_{quad,\beta}(\Gamma)  \subset H^1_{mol,\epsilon}(\Gamma) \cap L^2(\Gamma)$ and
$$\|f\|_{H^1_{mol\epsilon}} \lesssim \|f\|_{H^1_{quad,\beta}}\qquad \forall f\in H^1_{quad,\beta}(\Gamma).$$
\end{prop}

\begin{proof}
Let $f\in E^1_{quad,\beta}(\Gamma)$. We set
$$F(.,l) = [l\Delta]^\beta P^{l-1}f.$$
By definition of $E^1_{quad,\beta}(\Gamma)$, one has that $F \in T^1(\Gamma)$. Moreover, since $f\in L^2(\Gamma)$, $L^2$-boundedness of Littlewood-Paley functionals (see \cite{BRuss}, \cite{Fen1}) yields that $F \in T^2(\Gamma)$.
Thus, according to Theorem \ref{theoTent}, there exist a numerical sequence $(\lambda_i)_{i\in \N}$ and a sequence of $T^1$-atoms $(A_i)_{i\in \N}$ such that
$$F = \sum_{i=0}^\infty \lambda_i A_i \qquad \text{ in } T^1(\Gamma) \text{ and } T^2(\Gamma)$$
and
$$\sum_{i\in \N} |\lambda_i| \lesssim \|F\|_{T^1} = \|f\|_{H^1_{quad,\beta}}.$$
Choose $\eta$ as in Lemma \ref{molecules}. Since $f\in L^2(\Gamma)$,
\begin{equation} \label{L2DecomMol}\begin{split}
f & =  \pi_{\eta,\beta} F(.,l)  \\
& = \sum_{i = 0}^{+\infty} \lambda_i \pi_{\eta,\beta} (A_i)
  \end{split} \end{equation}
where the sum converges in $L^2(\Gamma)$ (see \cite[Corollary 2.3]{Fen2}). 
According to Lemma \ref{molecules}, $\pi_{\eta,\beta} (A_i)$ are molecules and 
then \eqref{L2DecomMol} would provide an $\epsilon$-molecular representation of $f$ if the convergence held in $L^1(\Gamma)$.
By uniqueness of the limit, it remains to prove that $\sum \lambda_i \pi_{\eta,\beta} (A_i)$ converges in $L^1$. 
Indeed,
\[\begin{split}
   \sum_{i\in \N} |\lambda_i| \left\|\pi_{\eta,\beta} (A_i)\right\|_{L^1} & \lesssim \sum_{i\in \N} |\lambda_i| \\
& < +\infty
 \end{split}\]
where the first line comes from Proposition \ref{Fen4sizemolec} and the second one from the fact that $(\lambda_i)_{i\in\N} \in \ell^1(\N)$.
\end{proof}

\begin{lem} \label{molecules2}
 Suppose that $A$ is a $T^1(\Gamma)$-atom associated with a ball $B \subset \Gamma$. 
Let $\epsilon>0$, there exist an integer $\eta = \eta_{\epsilon}$ and a uniform constant $C_{\epsilon}>0$ such that 
$C_{\epsilon}^{-1} d\Delta^{-\frac{1}{2}}\pi_{\eta,\frac{1}{2}}(A)$ is an $\epsilon$-molecule associated with the ball $B$.
\end{lem}

\begin{proof}
Let $d_0$ be the value in Proposition \ref{Fen4propDV}. Set again $\eta = \lceil \frac{d_0}{2} + \epsilon\rceil  + 2$. We will also write $t$ for $\lceil \frac{d_0}{2} + \epsilon \rceil \in \N^*$.

Let $A$ be a $T^1$-atom associated with a ball $B$ of radius $k$ and center $x_0$. We write
$$a : = d\Delta^{-\frac{1}{2}} \pi_{\eta,\frac{1}{2}}(A) = \sqrt k d (I+k\Delta)^{-\frac{1}{2}} b$$
where
\begin{equation}\label{bexpression}\begin{split} b & := \sum_{l\geq 1} \frac{c_l^\eta}{\sqrt{l}} \left(\frac{I+k\Delta}{k}\right)^{\frac{1}{2}}\Delta^{\eta-1}(I+P)^\eta P^{l-1} A(.,l) \\
   & = \frac1{\sqrt k}\sum_{m=0}^\infty a_m \sum_{l\geq 1} \frac{c_l^\eta}{\sqrt{l}} (I+ k\Delta) \Delta^{1+t} (I+P)^\eta P^{l+m-1} A(.,l) \\
  \end{split}\end{equation}
where $\sum a_m z^m$ is the Taylor serie of the function  $\ds \left(1+ k(1-z)\right)^{-\frac{1}{2}}$.

\noindent Let us check that $a$ is an $\epsilon$-molecule associated with $B$, up to multiplication by some harmless constant $C_{\epsilon}$. 

Let  $h \in L^2(x_0,2\eta^Bk)$. One has with the first equality in \eqref{bexpression},
\[\begin{split}
   \left|\left< b,h \right>\right| & \leq  k^{-\frac12}  \sum_{l\geq 1} \frac{c_l^\eta}{\sqrt{l}} \left|\left< A(.,l),\left(I+k\Delta\right)^{\frac{1}{2}}\Delta^{1+t}(I+P)^\eta P^{l-1} h \right>\right| \\
& \lesssim  k^{-\frac12} \sum_{l\geq 1} \frac{c_l^\eta}{\sqrt{l}} \|A(.,l)\|_{L^2(B)}\|\left(I+k\Delta\right)^{\frac{1}{2}}\Delta^{1+t}(I+P)^\eta P^{l-1}h\|_{L^2}\\ 
& \lesssim  \|A\|_{T^2_2} k^{-\frac12} \left(  \sum_{l= 1}^{k} l^{2(\eta-1)} \|\left(I+k\Delta\right)^{\frac{1}{2}}\Delta^{1+t}(I+P)^\eta P^{l-1}h\|_{L^2}^2  \right)^{\frac12} \\
& \lesssim  \|A\|_{T^2_2} k^{-\frac12} \left(  \sum_{l= 1}^{k} l^{2(1+t)} \|\left(I+k\Delta\right)^{\frac{1}{2}}\Delta^{1+t}(I+P)^\eta P^{l-1}h\|_{L^2}^2  \right)^{\frac12} \\
& \lesssim  \|A\|_{T^2_2} \|(I+P)^\eta h\|_{L^2} \\
& \lesssim  \|A\|_{T^2_2} \|h\|_{L^2}
 \end{split}\]
where we use that the functionals $\ds h \mapsto k^{-\frac12} \left( \sum_{l=1}^{k} l^{2(1+t)} |\left(I+k\Delta\right)^{M+\frac{1}{2}}\Delta^{1+t}P^{l-1}h|^2 \right)^{1/2}$ are $L^2$-bounded uniformly in $k$. 
Indeed, since $(-1) \notin \Sp(P)$, functional calculus provides, for some $a>-1$,
\[\begin{split} 
\|\left(I+k\Delta\right)^{\frac{1}{2}}\Delta^{1+t}P^{l-1}h\|_{L^2}^2 & = \int_{a}^1 (1+k(1-\lambda))(1-\lambda)^{2(1+t)} \lambda^{2(l-1)} dE_{hh}(\lambda)
\end{split}\]
Thus,
\[\begin{split}
   k^{-1}  & \sum_{l=1}^{k} l^{2(1+t)}  \|\left(I+k\Delta\right)^{\frac{1}{2}}\Delta^{1+t}P^{l-1}h\|_{L^2}^2 \\
& \lesssim \int_{a}^1 (1-\lambda)^{2(1+t)} \sum_{l=1}^{k} l^{2(1+t)-1}\lambda^{2(l-1)} dE_{hh}(\lambda) +  \int_{a}^1 (1-\lambda)^{2(1+t)+1} \sum_{l=1}^{k} l^{2(1+t)} \lambda^{2(l-1)} dE_{hh}(\lambda) \\
& \lesssim \int_{a}^1 (1-\lambda)^{2(1+t)} \sum_{l=1}^{\infty} l^{2(1+t)-1}\lambda^{2(l-1)} dE_{hh}(\lambda) +  \int_{a}^1 (1-\lambda)^{2(1+t)+1} \sum_{l=1}^{\infty} l^{2(1+t)} \lambda^{2(l-1)} dE_{hh}(\lambda) \\
& \lesssim \int_{a}^1 \frac{(1-\lambda)^{2(1+t)}}{(1-\lambda^2)^{2(1+t)}} dE_{hh}(\lambda) +  \int_{a}^1 \frac{(1-\lambda)^{2(1+t)+1}}{(1-\lambda^2)^{2(1+t)+1}} dE_{hh}(\lambda) \\
& \quad = \int_{a}^1 \left[(1+\lambda)^{-2(1+t)} + (1+\lambda)^{-2(1+t)-1} \right] dE_{hh}(\lambda) \\
& \lesssim \int_{a}^1 dE_{hh}(\lambda) = \|h\|_{L^2}^2
  \end{split}\]
  where the third inequality comes from the fact that $l^{\xi-1} \sim c^{\xi}_l$ (see Lemma B.1 in \cite{Fen1}).

\smallskip

Let $j> B\log_2(\eta) >1$ and $h\in L^2(C_j(B))$. One has $\rho(\Supp (I+P)^\eta h ,B) \gtrsim 2^jr$ (cf Lemma \ref{molecules}). The second identity in \eqref{bexpression} provides

\[\begin{split}
   \left|\left< b,g \right>\right| & \leq  k^{-\frac12} \sum_{m=0}^\infty a_m \sum_{l\geq 1} \frac{c_l^\eta}{\sqrt{l}} \left|\left< A(.,l),(I+k\Delta)\Delta^{1+t}(I+P)^\eta P^{l+m-1} h \right>\right| \\
& \lesssim k^{-\frac12} \sum_{m=0}^\infty a_m \sum_{l\geq 1} \frac{c_l^\eta}{\sqrt{l}} \|A(.,l-1)\|_{L^2(B)}\|\Delta^{1+t}(I+k\Delta)(I+P)^\eta P^{l+m-1}h\|_{L^2(B)}\\ 
& \lesssim k^{-\frac12} \|A\|_{T^2_2} \sum_{m=0}^\infty a_m \left(  \sum_{l= 1}^{k} l^{2(\eta-1)} \|(I+k\Delta)\Delta^{1+t}(I+P)^\eta P^{l+m-1}h\|_{L^2(B)}^2  \right)^{\frac12} \\
& \lesssim  k^{-\frac12} \|A\|_{T^2_2} \|(I+P)^\eta h\|_{L^2} \sum_{m=0}^\infty a_m \left(  \sum_{l= 1}^{k} l^{2(t+1)} \dfrac{k^2 + (l+m)^2}{(l+m)^{2(2+t)}} \left(1+ \frac{2^jk}{l+m}\right)^{-N}  \right)^{\frac12},
\end{split}\]
where the last line holds for any $N\in \N$ with Proposition \ref{GEforPk} (the constant depends on $N$). If we fix $N=2(2+t)$, we obtain
\[\begin{split}
 \left|\left< b,g \right>\right| & 
 \lesssim  k^{-\frac12} \|A\|_{T^2_2} \|(I+P)^\eta h\|_{L^2} \sum_{m=0}^\infty a_m  \left(  \sum_{l= 1}^{k} l^{2(t+1)}  k^2(2^jk)^{-2(2+t)} + (2^jk)^{-2(1+t)}  \right)^{\frac12}  \\
& \lesssim k^{-\frac12} \|A\|_{T^2_2} \|h\|_{L^2} (2^jk)^{-(1+t)}   \left(\sum_{m=0}^\infty a_m\right)  \left(  \sum_{l= 1}^{k} l^{2(1+t)} \right)^{\frac12} \\
& \lesssim  \|A\|_{T^2_2} \|h\|_{L^2} 2^{-j(1+t)} \\
& \lesssim  \frac{2^{-j(t+1)}}{V(B)^\frac12} \|h\|_{L^2} \\
& \lesssim \frac{2^{-j\epsilon}}{V(2^jB)} \|h\|_{L^2}
  \end{split}\]
where we used the relation $\sum a_m =1$ for the third line and the Proposition \ref{Fen4propDV} for the last one.
\end{proof}

\begin{prop} \label{Fen4Main22}
Let $M\in \N$ and $\epsilon>0$. Then $E^1_{quad,\frac{1}{2}}(T_\Gamma) \subset H^1_{mol,\epsilon}(T_\Gamma) \cap H^2(T_\Gamma)$ and
$$ \|G\|_{H^1_{mol,\epsilon}} \lesssim \|G\|_{H^1_{quad,\frac{1}{2}}} \qquad \forall G \in E^1_{quad,\frac{1}{2}}(T_\Gamma) $$
\end{prop}

\begin{proof}
Let $G\in E^1_{quad,\frac{1}{2}}(T_\Gamma)$. We set
$$F(.,l) = \sqrt{l+1} P^{l}d^*G.$$
By definition of $H^1_{quad,\frac{1}{2}}(T_\Gamma)$, one has that $F \in T^1(\Gamma)$. 
Moreover, since $G\in E^2(T_\Gamma)$, Proposition \ref{Fen4danddstar} yields that $\Delta^{-\frac{1}{2}}d^* G \in L^2(G)$ and therefore, with the $L^2$-boundedness of Littlewood-Paley functionals, $F\in T^2(\Gamma)$.

Thus, according to Theorem \ref{theoTent}, there exist a scalar sequence $(\lambda_i)_{i\in \N} \in \ell^1(\N)$ and a sequence of $T^1_2$-atoms $(A_i)_{i\in \N}$ such that
$$F = \sum_{i=0}^\infty \lambda_i A_i \qquad \text{ in } T^1(\Gamma) \text{ and in } T^2(\Gamma)$$
and
$$\sum_{i\in \N} |\lambda_i| \lesssim \|F\|_{T^1} = \|G\|_{H^1_{quad,\frac{1}{2}}}.$$
Choose $\eta$ as in Lemma \ref{molecules}. Using \cite[Corollary 2.3]{Fen2}, since $\Delta^{-\frac{1}{2}} d^* G\in L^2(\Gamma)$,
\[ \begin{split}
\Delta^{-\frac{1}{2}} d^* G& =  \pi_{\eta,\frac{1}{2}} F(.,l)  \\
& = \sum_{i = 0}^{+\infty} \lambda_i \pi_{\eta,\frac{1}{2}} (A_i)
  \end{split}\]
where the sum converges in $L^2(\Gamma)$.
Recall that $d \Delta^{-1} d^* = Id_{H^2(T_\Gamma)}$. Moreover, $d\Delta^{-\frac{1}{2}}$ is bounded from $L^2(\Gamma)$ to $L^2(T_\Gamma)$ (see Proposition \ref{Fen4danddstar}). Then
\begin{equation} \label{L2DecomMol2}
 G = \sum_{i=0}^{+\infty} \lambda_i d\Delta^{-\frac{1}{2}} \pi_{\eta,\frac{1}{2}} (A_i)
\end{equation}
where the sum converges in $L^2(T_\Gamma)$. 
According to Lemma \ref{molecules2}, $d\Delta^{-\frac{1}{2}} \pi_{M,\frac12} (A_i)$ are $\epsilon$-molecules and 
then \eqref{L2DecomMol2} would provide an $\epsilon$-molecular representation of $f$ if the convergence held in $L^1(\Gamma)$.
By uniqueness of the limit, it remains to prove that $\sum \lambda_i d\Delta^{-\frac{1}{2}}\pi_{\eta,\frac{1}{2}} (A_i)$ converges in $L^1$. 
Indeed,
\[\begin{split}
   \sum_{i\in \N} |\lambda_i| \left\|d\Delta^{-\frac12}\pi_{\eta,\frac{1}{2}} (A_i)\right\|_{L^1(T_\Gamma)} & \lesssim \sum_{i\in \N} |\lambda_i| \\
& < +\infty
 \end{split}\]
where the first line comes from Proposition \ref{Fen4sizemolec}  and the second one because $(\lambda_i)_{i\in\N} \in \ell^1(\N)$.
\end{proof}

\subsection{Completion of Hardy spaces.}

\begin{theo} \label{Main3}
 Let $\beta >0$. The completion $H^1_{quad,\beta}(\Gamma)$ of $E^1_{quad,\beta}(\Gamma)$ in $L^1(\Gamma)$ exists. 
 Moreover, if $\epsilon \in (0,+\infty)$, then the spaces $H^1_{mol,\epsilon}(\Gamma)$ and $H^1_{quad,\beta}(\Gamma)$ coincide. More precisely, we have 
 $$E^1_{quad,\beta}(\Gamma) = H^1_{BZ\kappa,M,\epsilon}(\Gamma) \cap L^2(\Gamma).$$
 \end{theo}
 
\noindent Once the equality $H^1_{mol,\epsilon}(\Gamma)= H^1_{quad,\beta}(\Gamma)$ is  established, this space will be denoted by $H^1(\Gamma)$.

\begin{proof} 

Let $\beta>0$ and $\epsilon>0$. Propositions \ref{Fen4Main11} and \ref{Fen4Main12} yield the equality of spaces
\begin{equation} \label{Fen4equalityE1E1} H^1_{mol,\epsilon}(\Gamma) \cap L^2(\Gamma) = E^1_{quad,\beta}(\Gamma) \end{equation}
with equivalent norms. In particular, $E^1_{quad,\beta}(\Gamma)\subset L^1(\Gamma)$.

Since the space of finite sum of $\epsilon$-molecules is dense in $H^1_{mol,\epsilon}(\Gamma)$ (see \cite[Lemma 4.5]{BZ} or \cite[Lemma 3.5]{Fen2}),
the space $H^1_{mol,\epsilon}(\Gamma)$ is the completion of $H^1_{mol,\epsilon}(\Gamma) \cap L^2(\Gamma)$ in $L^1$.
The completion $H^1_{quad,\beta}(\Gamma)$ of $E^1_{quad,\beta}(\Gamma)$ in $L^1(\Gamma)$ exists then too and satisfies
$$H^1_{quad,\beta}(\Gamma)=H^1_{mol,\epsilon}(\Gamma)$$
with equivalent norms.
\end{proof}

\begin{theo} \label{Main5}
 Let $\beta>0$. The completion $H^1_{quad,\beta}(T_\Gamma)$ of $E^1_{quad,\beta}(T_\Gamma)$ in $L^1(T_\Gamma)$ exists.
 
 Moreover, if $\epsilon \in (0,+\infty)$, then the spaces $H^1_{mol,\epsilon}(T_\Gamma)$ and $H^1_{quad,\beta}(T_\Gamma)$ coincide. More precisely, we have 
 $$E^1_{quad,\beta}(T_\Gamma) = H^1_{mol,\epsilon}(T_\Gamma) \cap H^2(T_\Gamma). $$
\end{theo}
\noindent Again, the space $H^1_{BZ2,M+\frac{1}{2},\epsilon}(T_\Gamma)=H^1_{quad,\beta}(T_\Gamma)$ will be denoted by $H^1(T_\Gamma)$.

\begin{proof} 
Let $\epsilon>0$. Propositions \ref{Fen4Main21} and \ref{Fen4Main22} yield the continuous embeddings
$$H^1_{mol,\epsilon}(T_\Gamma) \cap L^2(T_\Gamma) \subset E^1_{\frac12,\beta}(T_\Gamma) \subset  H^1_{mol,\epsilon}(T_\Gamma) \cap L^2(T_\Gamma)$$
from which we deduce the equality of the all the spaces, with equivalent norms.

It follows that the completion of $E^1_{quad,\frac12}(T_\Gamma)$ in $L^1(T_\Gamma)$ exists and satisfies, 
since $H^1_{mol,M-\frac12,p,\epsilon}(T_\Gamma)$ is the completion of $H^1_{mol,M-\frac12,p,\epsilon}(T_\Gamma) \cap L^2(T_\Gamma)$ in $L^1(T_\Gamma)$, we deduce as in the proof of Theorem \ref{Main3} that
$$H^1_{quad,\frac12}(T_\Gamma)=H^1_{mol,\epsilon}(T_\Gamma)$$
with equivalent norms.

Moreover, notice that if $F \in H^2(T_\Gamma)$,
\begin{equation} \label{Fen4trucmuch} F \in E^1_{quad,\beta}(T_\Gamma)  \Longleftrightarrow  \Delta^{-\frac{1}{2}}d^*F \in  E^1_{quad,\beta}(\Gamma).\end{equation}
Indeed,  the implication $\Delta^{-\frac{1}{2}}d^*F \in  E^1_{quad,\beta}(\Gamma)\Rightarrow F \in E^1_{quad,\beta}(T_\Gamma)$ is obvious, and the converse is due to Proposition \ref{Fen4danddstar}.  
Theorem \ref{Main3} implies that all the spaces $E^1_{quad,\beta}(\Gamma)$, $\beta>0$, coincide. 
Together with \eqref{Fen4trucmuch}, for all $\beta>0$, the spaces $E^1_{quad,\beta}(T_\Gamma)$ coincide with $E^1_{quad,\frac12}(T_\Gamma)$.
Hence, for all $\beta>0$, the completion $H^1_{quad,\beta}(T_\Gamma)$ of $E^1_{quad,\beta}(T_\Gamma)$ in $L^1(T_\Gamma)$ exists and satisfies
$$H^1_{quad,\beta}(T_\Gamma)=H^1_{mol,\epsilon}(T_\Gamma)$$
with equivalent norms.
\end{proof}

\section{Proof of Theorem \ref{MainTheo}}

\label{ProofMainTheo}

Let $(\Gamma,\mu)$ be a weighted graph as defined in Section \ref{defgraphs}. Assume that $\Gamma$ satisfies \eqref{LB}, \eqref{DV} and \eqref{UE}. 

\medskip

Let $H^1(\Gamma)$ and $H^1(T_\Gamma)$ by the two spaces defined with Theorems \ref{Main3} and \ref{Main5}. 
The Theorem \ref{MainTheo} will be proven if we prove the four following facts.
\begin{enumerate}[(1)]
\item The spaces $H^1(\Gamma)$ and $H^1(T_\Gamma)$ are included in $L^1(\Gamma)$ and $L^1(T_\Gamma)$ respectively.
\item The operator $d\Delta^{-\frac12}$ satisfies
$$\|d\Delta^{-\frac12}f\|_{H^1(T_\Gamma)} \simeq \|f\|_{H^1(\Gamma)}  \qquad \forall f\in H^1(\Gamma) \cap L^2(\Gamma).$$
\item Let $\mathcal O$ be a $L^2$-bounded linear operator, which is also bounded from $H^1(\Gamma) \cap L^2(\Gamma)$ to $L^1(\Gamma)$. Then
$$\|\mathcal O f\|_{L^p(\Gamma)} \leq C \|f\|_{L^p(\Gamma)} \qquad \forall f\in L^p(\Gamma) \cap L^2(\Gamma),$$
that is $\mathcal O$ can be extended to a linear operator bounded on $L^p(\Gamma)$.
\item  The Riesz transform $\nabla \Delta^{-\frac12}$ satisfies
$$\|\nabla \Delta^{-\frac12} O f\|_{L^p(\Gamma)} \leq C \|f\|_{L^p(\Gamma)} \qquad \forall f\in L^p(\Gamma) \cap L^2(\Gamma),$$
and thus can be extended to a $L^p(\Gamma)$-bounded operator.
\end{enumerate}

The fact (1) is given by the fact that $H^1(\Gamma) = H^1_{mol,\epsilon}(\Gamma)$ and $H^1(T_\Gamma) = H^1_{mol,\epsilon}(T_\Gamma)$ are continuously embedded in $L^1(\Gamma)$ and $L^1(T_\Gamma)$ respectively (see Proposition \ref{Fen4H1complete}).

\medskip

The fact (2) comes from the definition of the Hardy spaces via quadratic functionals. Indeed, by definition of the norm in $E^1_{quad,\beta}(T_\Gamma)$, we have, for any $f\in H^1(\Gamma) \cap L^2(\Gamma)$
\[\begin{split}
\|d\Delta^{-\frac12}f\|_{H^1(T_\Gamma)} & \simeq \|d\Delta^{-\frac12}f\|_{H^1_{quad,\beta}(T_\Gamma)} \\
& \quad =  \|\Delta^{-\frac12}d^*d\Delta^{-\frac12}f\|_{H^1_{quad,\beta}(\Gamma)} \\
& \quad = \|f\|_{H^1_{quad,\beta}(\Gamma)} \\
& \simeq  \|f\|_{H^1(\Gamma)} \\
\end{split}\]

\medskip

Check that Proposition \ref{Fen4Main32} yields the inclusion $E^1_0(\Gamma) \subset H^1(\Gamma)$. Recall that fact (1) says that 
$H^1(\Gamma) \subset L^1(\Gamma)$. Therefore, Corollary \ref{Fen4interpolationH1L2bis} implies exactly fact (3).

\medskip

It remains to verify fact (4). Note that it will be a direct consequence of facts (1),  (2) and (3) if the operator $\nabla \Delta^{-\frac12}$ were linear.  

\noindent Facts (1) and (2) yields
\begin{equation} \label{Fen4TRisL1bounded}  \|\nabla \Delta^{-\frac{1}{2}} f\|_{L^1(\Gamma)} = \|d\Delta^{-\frac{1}{2}} f\|_{L^1(T_\Gamma)} \lesssim \|d\Delta^{-\frac12}f\|_{H^1(T_\Gamma)} \lesssim \|f\|_{H^1(\Gamma)}. \end{equation}
Define now for any function $\phi$ on $T_\Gamma$ the linear operator
$$\nabla_\phi f(x) = \sum_{y\in \Gamma} p(x,y) df(x,y) \phi(x,y)m(y).$$
The boundedness \eqref{Fen4TRisL1bounded} yields then the estimate
$$\|\nabla_\phi \Delta^{-\frac12} f\|_{L^1} \leq \|\nabla \Delta^{-\frac{1}{2}} f\|_{L^1(\Gamma)}  \sup_{x\in \Gamma} |\phi(x,.)|_{T_x} \lesssim \|f\|_{H^1} \sup_{x\in \Gamma} |\phi(x,.)|_{T_x}.$$
Moreover, since $\|\nabla \Delta^{-\frac12}f\|_{L^2} = \|f\|_{L^2}$ (cf \eqref{Fen4Katopb}), one has
$$\|\nabla_\phi \Delta^{-\frac12} f\|_{L^2} \lesssim \|f\|_{L^2} \sup_{x\in \Gamma} |\phi(x,.)|_{T_x}.$$
Set $\ds \phi_f(x,y) = \frac{df(x,y)}{\nabla f(x)}$. Note that $\ds \sup_{x\in \Gamma} |\phi_f(x,.)|_{T_x} =1$.
With fact (3), one has then
$$\|\nabla \Delta^{-\frac12} f\|_{L^p} = \|\nabla_{\phi_f} \Delta^{-\frac12} f\|_{L^p}.\lesssim \|f\|_{L^p} $$

\appendix

\section{Examples of graph satisfying \eqref{DV} and \eqref{UE}}

\label{Fen4SectionExamples}

The aim of this appendix is to give some examples of graphs satisfying our three conditions \eqref{LB}, \eqref{DV} and \eqref{UE}. Indeed, if the condition \eqref{LB} is classical and can be seen as a condition of analyticity, the conditions \eqref{DV} and \eqref{UE} can look unnatural, in particular when $\beta$ is not a constant.

\medskip

But first, let us recall some classical estimates. 
The most common estimates on the Markov kernel $p_k(x,y)$ are the Gaussian estimates, given by 
\begin{equation} \tag{UE$_2$}
p_k(x,y) \lesssim \frac{1}{V_{\d}(x,\sqrt{k}} \exp\left(-c \frac{\d(x,y)^2}{k} \right) \qquad \forall x,y\in \Gamma, \, \forall k \in \N^*.
\end{equation}
These estimates are satisfied by the Markov kernel when the operator $P$ is a random walk on $\Z^n$, or more generally on a discrete group of polynomial growth (see \cite[Theorem 5.1]{HebSC}). The classical result of Delmotte \cite[Theorem 1.7]{Delmotte1} states that the upper and lower Gaussian estimates on graphs
\begin{equation} \tag{LUE$_2$}\frac{c}{V_{\d}(x,\sqrt{k}} \exp\left(-C \frac{\d(x,y)^2}{k} \right) \leq p_k(x,y) \leq \frac{C}{V_{\d}(x,\sqrt{k}} \exp\left(-c \frac{\d(x,y)^2}{k} \right)\end{equation}
are equivalent to the conjunction of \eqref{LB}, \eqref{DV2} and the $L^2$-Poincar\'é inequality on balls (and are also equivalent to a parabolic Harnack inequality).
\noindent Besides, the upper Gaussian estimate \eqref{UEd2} can hold when the corresponding Gaussian lower bounds are not satisfied (see for example two copies of $\Z^n$ linked by an edge in \cite{Russ}).

\medskip

There exist however some graphs where the Gaussian estimates \eqref{UEd2} don't hold. It is the case of the Sierpinski gaskets (\cite{BarlowStFlour}) or the Vicsek graphs (\cite{BCG}). This kind of graphs, whose behaviors is strikingly different from the graphs $\Z^n$, are called fractals graphs. On these graphs, the Markov kernel satisfy some so-called "sub-Gaussian" estimates
\begin{equation} \label{UEm} \tag{UE$_m$}
p_{k-1}(x,y) \leq \frac{C}{V_\d(x, k^\frac1m)} \exp\left( -c \left[ \frac{\d(x,y)^m}{k}\right]^\frac{1}{m-1} \right) \qquad \forall x,y\in \Gamma, \, \forall k \in \N^*
\end{equation}
where $m> 2$ is a real constant. Note that \eqref{UEm} is incompatible with \eqref{UEd2}. Indeed, with \eqref{DV2}, the estimates \eqref{UEd2} and \eqref{UEm} yields respectively the on-diagonal lower bound $p_{2k}(x,x) \geq \frac{c}{V_\d(x, \sqrt k)}$ and $p_{2k}(x,x) \geq \frac{c}{V_\d(x, k^\frac1m)}$ (see Proposition \ref{ODprop} below). Therefore the conjunction of \eqref{UEd2} and \eqref{UEm} provides a contradiction with the doubling volume property.

Necessary and sufficient conditions for \eqref{UEm} and the corresponding lower bounds to hold have been given in, for instance, \cite{GrigorT}, \cite{BCK05} and \cite{BGK}. Some results on sub-Gaussian estimates are also collected in  \cite{grigofractal}.

\medskip

In the sequel, we say that $\Gamma$ satisfies \eqref{VD} if
\begin{equation} \label{VD} \tag{V$_D$}
V_{\d}(x,r) \simeq r^D \qquad \forall x\in \Gamma, \, \forall r\in \N^*.
\end{equation}
Clearly, \eqref{VD} implies \eqref{DV2}. It is shown by Barlow in \cite{BarlowWV} that 
\begin{theo} \label{Barlow}
For any real $D\geq 1$ and $2\leq m\leq D+1$, there exists a graph satisfying \eqref{VD} and \eqref{UEm}. Moreover, the range of $D$ and $m$ are sharp.
\end{theo}
For example, the classical Sierpinski gasket 
satisfies \eqref{VD} and \eqref{UEm} for $D= \log_2 3$ and $m= \log_2 5$. 

\medskip

The Gaussian and sub-Gaussian estimates \eqref{UEd2} and \eqref{UEm} are stronger than the estimates \eqref{UE} when $\beta \equiv 2,m$. It is only a consequence of the fact that an exponential decay is faster than a polynomial decay. Therefore, our results work for the two cases of graphs that satisfies the Gaussian and sub-Gaussian estimates \eqref{UEd2} and \eqref{UEm}.
However, we will describe below a case of graph that satisfies neither \eqref{UEd2} nor \eqref{UEm} and yet satisfies our condition. That is why our assumptions have been expanded  to the case where $\beta$ is not a constant. We do not think that there exists a graph satisfying \eqref{UE} for any bounded function $\beta$ that satisfy our condition - for example, there are no graphs satisfying \eqref{UE} with $\beta \equiv 1$ - but we didn't find a way to prove that the function $\beta$ has to be greater than 2.

In the cases where $\beta$ is a constant (Gaussian and sub-Gaussian cases), it is easy to check that the condition \eqref{DV} is equivalent to the doubling property
\begin{equation} \tag{DV2}
V_\d(x,2r) \leq C V_\d(x,r).
\end{equation}
However, the condition \eqref{DV2} is not the good "doubling" property, because the balls in a graph with the metric given by $\rho := \d^\beta$ can be ellipsoids for the same graph with the metric given by $\d$. That is why the doubling condition in our case depends also on the parameter $\beta$. 

\medskip

We present now a graph $\Gamma$ that satisfies \eqref{DV} and \eqref{UE} for some function $\beta$, but that doesn't satisfy \eqref{UEm} for any real $m$. In order to do this, we will build the graph $\Gamma$ as a product of two graphs $\Gamma_1$ and $\Gamma_2$ that satisfies \eqref{UEm} for different values $m_1$ and $m_2$. 
A more general discussion about the fact that the off-diagonal decay of the Markov kernel may depend on the direction can be found in \cite{HK04,HK04b}.

\begin{defi}
Let $(\Gamma_1,\mu^1)$ and $(\Gamma_2,\mu^2)$ be two weighted graphs. 
The graph $(\Gamma,\mu)$ is the free product of $\Gamma_1$ and $\Gamma_2$ if
\begin{enumerate}[(i)]
 \item $\Gamma = \Gamma_1 \times \Gamma_2$,
 \item for all $x=(x_1,x_2)\in \Gamma$ and $y=(y_1,y_2) \in \Gamma$, $\mu_{xy} = \mu^1_{x_1y_1}\mu^2_{x_2y_2}$,
\end{enumerate}
Note that in this case, if $\d_1$ and $\d_2$ are the canonic distances on $(\Gamma_1, \mu_1)$ and  $(\Gamma_2, \mu_2)$ respectively, then the distance on $\Gamma$ is $\d(x,y): = \max\{\d_1(x_1,y_1), \d_2(x_2, y_2)\}$.
\end{defi}

\begin{rmk}
Let $(\Gamma,\mu)$ be the free product of $(\Gamma_1,\mu^1)$ and $(\Gamma_2,\mu^2)$.
Then the following facts are satisfied for any vertices $x=(x_1,x_2)$ and $y=(y_1,y_2)$
\begin{enumerate}[(i)]
 \item $x \sim y$ if and only if $x_1\sim y_1$ and $x_1\sim y_1$.
 \item $\d(x,y) = \max \{\d_1(x_1,y_1),\d(x_2,y_2)\}$, where $\d$, $\d_1$ and $\d_2$ are the canonical distances on $\Gamma$, $\Gamma_1$ and  $\Gamma_2$ respectively.
\end{enumerate}
\end{rmk}

\begin{prop} \label{Fen4Hypfreeproduct}
Let $(\Gamma_1,\mu^1)$ and $(\Gamma_2,\mu^2)$ satisfying \eqref{LB}.  Set $(\Gamma,\mu)$ be the graph defined as the free product of $\Gamma_1$ and $\Gamma_2$. Then the graph $(\Gamma,\mu,\rho)$ satisfies \eqref{LB}.

Moreover, let $\beta_1, \beta_2$ be functions bounded from below by 1 and from above by $B$. If $\Gamma_1$ and $\Gamma_2$ satisfy respectively (D$_{\beta_1}$), (UE$_{\beta_1}$), (D$_{\beta_2}$) and (UE$_{\beta_2}$), then $\Gamma$ satisfy \eqref{DV} and \eqref{UE} with 
\begin{equation} \label{betaformula}
\beta : = \sup\left\{\beta_1 \frac{\ln \d_1}{\ln \d}, \beta_2 \frac{\ln \d_2}{\ln \d}\right\}.
\end{equation}
when $\d \geq 2$ and $\beta = 1$ when $\d \leq 1$.
\end{prop}

\begin{proof}
Let $p$, $p_1$, $p_2$, $m$, $m_1$, $m_2$ the measures and the Markov kernel of respectively the graphs $\Gamma$, $\Gamma_1$ and $\Gamma_2$.

Remark that $p(x,x) = p^1(x_1,x_1)\cdot p^2(x_2,x_2)$ and $m(x) = m_1(x_1)\cdot m_2(x_2)$. Thus, the fact that $\Gamma$ satisfies \eqref{LB} if and only if both $\Gamma_1$ and $\Gamma_2$ satisfies \eqref{LB} is immediate.

\medskip

Check that by definition of $\beta$, one has $1\leq \beta \leq \max\{\beta_1, \beta_2\}$ and
$$\d^\beta = \sup \{\d_1^{\beta_1}, \d_2^{\beta_2}\}.$$
Set $\rho:=\d^\beta$, $\rho_1: = \d_1^{\beta_1}$ and $\rho_2: = \d_2^{\beta_2}$.
By construction, one has $B_\rho(x,k) = B_{\rho_1}(x_1,k) \times B_{\rho_2}(x_2,k)$. As a consequence,
$$V_\rho(x,k) = V_{\rho_1}(x_1,k) V_{\rho_2}(x_2,k)$$
and then assertion \eqref{DV} follows from the assumptions (D$_{\beta_1}$) and (D$_{\beta_2}$) on the graphs $\Gamma_1$ and $\Gamma_2$.

Recall that we have by construction $p(x,y) = p^1(x_1,y_1)p^2(x_2,y_2)$. Therefore, by induction, we get the relation $p_k(x,y) = p^1_k(x_1,y_1)p^2_k(x_2,y_2)$.
Consequently, for any $N\in \N$, 
\[\begin{split}
   p_{k-1}(x,y) & \leq \frac{C_N}{V_{\rho_1}(x_1,k) V_{\rho_2}(x_2,k)} \left(1+ \frac{\rho_1(x_1,y_1)}{k} \right)^{-N} \left(1+ \frac{\rho_2(x_2,y_2)}{k} \right)^{-N} \\
& \leq C_N \frac{1}{V_{\rho}(x,k)} \min \left\{ \left(1+ \frac{\rho_1(x_1,y_1)}{k} \right)^{-N}, \left(1+ \frac{\rho_2(x_2,y_2)}{k} \right)^{-N} \right\} \\
& \quad = C_N  \frac{1}{V_{\rho}(x,k)} \left(1+ \frac{\max\{\rho_1(x_1,y_1), \rho_2(x_2,y_2)\}}{k} \right)^{-N} \\
& \quad = C_N  \frac{1}{V_{\rho}(x,k)} \left(1+ \frac{\rho(x,y)}{k} \right)^{-N}.
  \end{split}\]
Thus $\Gamma$ satisfies \eqref{UE}.
\end{proof}

We need also the classical following result.

\begin{prop} \label{ODprop}
Let $\Gamma$ be a graph satisfying \eqref{DV} and \eqref{UE} for some bounded function $\beta \geq 1$. Set $\rho:= \d^\beta$. Then
\begin{equation} \label{ODestimate}
p_{2k}(x,x) \simeq \frac{1}{V_\rho(x,k)} \qquad \forall k\in \N^*,\, \forall x \in \Gamma.
\end{equation}
\end{prop}

\begin{proof}
The proof is similar to the one of \cite[Theorem 6.1]{CoulGrigor}, which establish the result when the graph satisfies \eqref{DV2} and \eqref{UEd2}. We will do it again for completeness.

\medskip

The upper estimate in \eqref{ODestimate} follows directly from \eqref{DV} and \eqref{UE}. 

We turn to the on-diagonal lower estimate. Let $d$ be the constant that appears in Proposition \ref{Fen4propDV}. With \eqref{UE}, one has
\[\begin{split}
\sum_{j\geq j_0} \sum_{y\in C_j(x,k)} p_{k}(x,y) m(y) 
& \lesssim  \sum_{j\geq j_0}\sum_{y\in C_j(x,k)} \frac{1}{V_\rho(x,k)} 2^{-j(d+1)} m(y) \\
& \lesssim \sum_{j\geq j_0} \sum_{y\in C_j(x,k)} \frac{1}{V_\rho(x,2^{B+j+1}k)} 2^{-j} m(y) \\
& \lesssim \sum_{j\geq j_0} 2^{-j} \lesssim 2^{-{j_0}}
\end{split}\]
with a constant that doesn't depend on $j_0$, $x\in \Gamma$ or $k\in \N^*$. Thus can can fix $j_0$ such that 
$$\sum_{j\geq j_0} \sum_{y\in C_j(x,k)} p_{k}(x,y) m(y)  \leq \frac12 \qquad \forall x\in \Gamma, \, \forall k\in \N^*.$$
Set $B=B_\rho(x,2^{j_0+B})$. Since $\ds \sum_{y\in \Gamma} p(x,y) m(y) = 1$, one has
$$\sum_{y\in B} p(x,y) m(y) \leq \frac12$$
and thus
\[\begin{split}
p_{2k}(x,x) & = \sum_{y\in \Gamma} p_k(x,y)^2 m(y)  \geq \sum_{y\in B} p_k(x,y)^2 m(y) \\
& \geq \frac{1}{V_\rho(B)} \sum_{y\in B} p_k(x,y) m(y)  \geq \frac{1}{2V_\rho(B)}.
\end{split}\]
We conclude by remarking that \eqref{DV} implies 
$$\frac{1}{V_\rho(B)} \gtrsim \frac{1}{V_\rho(x,k)}.$$
\end{proof}

\begin{prop} \label{impossible}
Let $(\Gamma_1, \mu^1)$ satisfying \eqref{LB}, (V$_{D_1}$) and (UE$_{m_1}$) and $(\Gamma_2, \mu^2)$ satisfying \eqref{LB}, (V$_{D_2}$) and (UE$_{m_2}$).
Then the free product graph $(\Gamma,\mu)$ satisfies \eqref{UEm} if and only if $m=m_1=m_2$.
\end{prop}

\begin{proof} Assume that $(\Gamma,\mu)$ satisfies $\eqref{UEm}$
 for some $m$. Moreover, according to Proposition \ref{Fen4Hypfreeproduct}, $\Gamma$ satisfies \eqref{UE} with $\beta$ satisfying
$$\rho:= \d^\beta = \sup\{\d_1^{m_1}, \d_2^{m_2}\}.$$
Therefore, with Proposition \ref{ODprop}, one has
$$p_{2k}(x,x) \simeq \frac{1}{V_\rho(x,k)} \simeq \frac{1}{V_{\d}(x,k^{1/m})}.$$
Note that 
$$V_\rho(x,k) = V_{\d_1}(x_1,k^{1/m_1}) V_{\d_2}(x_2,k^{1/m_2}) \simeq k^{\frac{D_1}{m_1}+ \frac{D_2}{m_2}} $$
 and
$$V_\d(x,k^{1/m}) = V_{\d_1}(x_1,k^{1/m}) V_{\d_2}(x_2,k^{1/m}) \simeq k^{\frac{D_1+D_2}{m}}.$$
Therefore, $m$ can only be 
\begin{equation} \label{mrelation}
m = \dfrac{D_1+D_2}{\frac{m_1}{D_1}+ \frac{m_2}{D_2}}.
\end{equation}

Without loss of generality, we can choose $m_1 \leq m_2$. In this case, we have $m_1 \leq m$. Since $\Gamma$ satisfies \eqref{UEm}, then we have for any $x,y,z \in \Gamma_1^2 \times \Gamma_2$,
\[\begin{split}
p_{2k}((x,z),(y,z)) & = p_{2k}^1(x,y) p_{2k}^2(z,z) \\
& \lesssim k^{-\frac{D_1}{m_1}- \frac{D_2}{m_2}} \exp\left(-c\left[\frac{\d_1^m(x,y)}{2k}\right]^{\frac1{m-1}}\right).
\end{split}\]
Yet, since $p_{2k}(z,z) \simeq k^{- \frac{D_2}{m_2}}$ (with Proposition \eqref{ODprop}), one has
\[\begin{split}
p_{2k}^1(x,y) 
& \lesssim k^{-\frac{D_1}{m_1}} \exp\left(-c\left[\frac{\d_1^m(x,y)}{2k}\right]^{\frac1{m-1}}\right). \\
& \lesssim k^{-\frac{D_1}{m}} \exp\left(-c\left[\frac{\d_1^m(x,y)}{2k}\right]^{\frac1{m-1}}\right)
\end{split}\]
and so $\Gamma_1$ satisfies \eqref{UEm}. It implies, again with Proposition \eqref{ODprop}, that for any $k\in \N^*$, it holds
$$p_{2k}(x,x) \simeq k^{-\frac{D_1}{m_1}} \simeq k^{-\frac{D_1}{m}}.$$
The last fact is possible only if $m=m_1$. At last, the relation \eqref{mrelation} allows us to say that $m_2 = m = m_1$.
\end{proof}

\begin{cor} \label{CorCex}
There exists a graph $\Gamma$ that satisfies \eqref{LB}, \eqref{DV} and \eqref{UE} for some bounded function $\beta\geq 1$ but that, for any constant $m\geq 2$, doesn't satisfies the combination of \eqref{DV2} and \eqref{UEm}.
\end{cor}

\begin{proof}
Theorem \ref{Barlow} yields that there exist $\Gamma_1$ that satisfies (V$_2$) and (UE$_2$) and $\Gamma_2$ that  that satisfies (V$_2$) and (UE$_3$). 
Set $\Gamma$ the free product of $\Gamma_1$ and $\Gamma_2$.
According to Proposition \ref{impossible}, there doesn't exist any constant $m\geq 2$ such that $\Gamma$ satisfies both \eqref{DV2} and \eqref{UEm}.
\end{proof}

\section{The case of Riemannian manifolds}

\label{Riemannian}

\subsection{Results}

Let $M$ be a Riemannian manifold. Denotes by $m$ the Riemannian measure et by $\d$ the Riemannian distance. The notation $T^* M$ is used for the tangent bundle of $M$.
The operator $d$ denotes the exterior differentiation on $M$ and $d^*$ is the adjoint of $d$. That is $d$ maps functions on 1-forms and $d^*$ maps 1-forms on functions. Set $\Delta = d^*d$ the positive Laplace Beltrami operator. As in graphs, we will use $\nabla f(x)$ for the length of the gradient $|d\, f(x)|$.

Let $\beta: M^2 \mapsto \R$ such that $1\leq \beta \leq B <+\infty$. We set $\rho:=\d^\beta$, $B(x,t) = B_{\d^\beta}(x,t) = \{y\in M, \, \rho(x,y) < t\}$ and
$V(x,t)=V_{\d^\beta}(x,t) = m(B(x,t))$. 
The following assumptions will be assumed throughout this section:
\begin{itemize} 
\item The space is doubling for $\rho$, that is
\begin{equation} \label{DVonM} \tag{D$_\beta$}
V(x,2t) \leq V(x,t) \qquad \forall x\in M, \, \forall t>0.
\end{equation}
\item The operator $\Delta$ generates an analytic semigroup $H_t:=e^{-t\Delta}$. The semigroup $H_t$ has a positive kernel $h_t$ satisfying: for all $N\in \N$, there exists $C_N >0$ such that
\begin{equation} \label{UEonM} \tag{UE$_\beta$}
h_t(x,y) \leq C_N \left( 1 + \frac{\rho(x,y)}{t} \right)^{-N}.
\end{equation}
\end{itemize}

Under these assumptions, we can obtain the same assumptions as in the case of graphs, that is

\begin{theo} \label{MainTheoonM}
Let $M$ be a connected non-compact Riemannian manifold. Assume that $M$ satisfies \eqref{DVonM} and \eqref{UEonM}.
Then there exists two complete spaces $H^1(M) \subset L^1(M)$ and $H^1(T^*M) \subset L^1(T^*M)$ such that:
\begin{enumerate}
\item The Riesz transform $d\Delta^{-\frac12}$ is an homomorphism between $H^1(M)$ and $H^1(T^*M)$.
\item Every linear operator bounded from $H^1(M)$ to $L^1(M)$ and bounded on $L^2(M)$ is bounded on $L^p(M)$ for any $p\in (1,2)$.
\end{enumerate}
As a consequence, the Riesz transform $\nabla \Delta^{-\frac12}$ is bounded on $L^p(M)$ for any $p\in (1,2)$.
\end{theo}

\subsection{Properties of the space $H^1(M)$}

The spaces $H^1(M)$ and $H^1(T^*M)$ in Theorem \ref{MainTheoonM} satisfies similar properties as $H^1(\Gamma)$ and $H^1(T_\Gamma)$. For short, we will state only the ones of $H^1(M)$.

\begin{defi} \label{Fen4molec0onM}
Let $\epsilon \in (0,+\infty)$. A function $a\in L^2(M)$ is called a $\epsilon$-molecule if there exist $x\in M$, $t>0$ and a function $b \in  L^2(M)$ such that
\begin{enumerate}[(i)]
 \item $ a = [I-(I+t\Delta)^{-1}]b$, 
 \item $\ds \|b\|_{L^2(C_j(x,t))} \leq 2^{-j\epsilon} V(x,2^jt)^{-\frac{1}{2}}$, for all $j\geq 0$.
\end{enumerate}
\end{defi}

Here and after, $C_j(x,t)$ denotes some annulus of center $x$, of small radius $2^jt$ and of big radius $2^{j+1}t$.

\begin{prop}
Let $\epsilon \in (0,+\infty)$. The $\epsilon$-molecules are uniformly bounded in $L^1(M)$. Moreover, any function $f \in H^1(M)$ admits an $\epsilon$-molecular representation, that is if there exist a sequence $(\lambda_i)_{i\in \N} \in \ell^1$ and a sequence $(a_i)_{i\in \N}$ of $\epsilon$-molecules such that
\begin{equation} \label{Fen4sumfonM}
f=\sum_{i=0}^\infty \lambda_i a_i
\end{equation}
where the convergence of the series to $f$ holds in $L^1(M)$. Moreover the $(\lambda_i)_{i\in \N}$ can be chosen such that $\sum_i |\lambda_i| \simeq \|f\|_{H?1(M)}$. \end{prop}

The space $H^1(M)$ can also be defined via quadratic functionals.

\begin{defi}
 Define, for $\alpha>0$, the quadratic functional $L_\alpha$ on $L^2(\Gamma)$ by
$$L_\alpha f(x) =  \left( \int_{\gamma(x)} \frac{1}{ t V(x,t)} |(t\Delta)^{\alpha} H_t f(y)|^2 dm(y) dt \right)^{\frac{1}{2}}$$
where $\gamma(x) = \left\{ (y,t) \in M \times (0,+\infty), \, \rho(x,y) < t \right\}$.
\end{defi}

\begin{prop}
Let $\alpha >0$. One has the equivalence
$$\|L_\beta f\|_{L^1(M)} \simeq \|f\|_{H^1(M)}$$
once $f\in L^2(M)$ and one of the two quantities is finite. In particular, the space 
$$E^1_{quad,\beta}(M) : = \left\{ f\in L^2(M), \,  \|L_\beta f\|_{L^1(M)} < +\infty \right\}$$
can be completed in $L^1(M)$.
\end{prop}

In the case where $\beta \equiv 2$, the space $H^1(M)$ constructed here is the same as the one in \cite{DY2}, \cite{AMR} or \cite{HLMMY}. In particular, the space $H^1(M)$ has a functional calculus and an atomic decomposition. 
When $\beta \equiv m$ with $m>2$ is a constant, our space $H^1(M)$ appears to be the same as the ones in \cite{KU} and \cite{ChenThesis}. In these two last references, $H^1(M)$ is defined as the completed space of $E^1_{quad,1}(M)$ and has also a molecular decomposition. The molecular decomposition of \cite{KU} and \cite{ChenThesis} is different of ours, but yields the same space. Indeed, in the two cases, the authors proved that $H^1(M)$ is the completed space of $E^1_{quad,1}(M)$.
Note also that we gave in the present paper some "weak functional calculus" on $H^1(M)$ since the Hardy spaces $H^1(M)$ can be defined from any quadratic functional $L_\alpha$, $\alpha>0$, and not only from $L_1$. This "weak functional calculus" is a crucial point in our proof. One can wonder whether there is the same functional calculus on the spaces $H^1(M)$ as the one found in \cite{AMR}. We didn't know the answer of this question.

\subsection{Discussion on the proofs}

All the methods used here have they counterparts in the case of Riemannian manifolds. Most of them can be found in \cite{HLMMY}. Let us emphasize only two particular points of the proof in the continuous case.

First, the key point argument of this article, that is the Stein relation used in Proposition \ref{Fen4GEforNablaPk} can be also done in the case of Riemannian manifold. It is actually easier to prove because the results from \cite{Dungey2} and \cite{Fen1} on the pseudo-gradient are not needed. 
In the case of Riemannian manifolds, a result similar to Proposition\ref{Fen4GEforNablaPk} can be found in \cite[Lemma 2.2]{CCFR}.

The second point is on the proof of Proposition \ref{Fen4Main11}. We used the $L^1$-boundedness of $\Delta$ to prove that 
\begin{equation} \label{LsuminsumL}
\left\|L_\alpha \left(\sum_{i\in \N} \lambda_i a_i\right)\right\|_{L^1} \leq  \sum_{i\in \N} |\lambda_i|  \|L_\alpha a_i\|_{L^1}.
\end{equation}
In the case of Riemannian manifold, the $L^1$-boundedness of $\Delta$ is replaced by the $L^1$-boundedness of the semigroup $\Delta^\alpha H_t$ for $t>0$. Indeed, the pointwise estimates \eqref{UEonM} and the $L^2$-analyticity of $H_t$ yields some pointwise estimates on the kernel of $\Delta H_t$, which implies in return the $L^1$-analyticity of $H_t$. In particular, $\Delta^\alpha H_t$ is $L^1$-bounded for any $t>0$ and $\alpha>0$.

\subsection{Manifolds satisfying \eqref{DVonM} and \eqref{UEonM}}

As in the case of graph, the function $\beta$ has probably to satisfy more properties than the ones assumed (we only need $\beta$ bounded from below by 1 and from above by some constant $B$). Yet, the aim of the article is not to find the sets of $\beta$ that can actually occur.

A manifold $M$ can be built from graph $\Gamma$ by replacing the edges of the graph with tubes of length 1 and then gluing the tubes together smoothly at the vertices. In this case, $M$ and $\Gamma$ will have similar structures at infinity (see the Appendix in \cite{CCFR}). Together with Theorem \ref{Barlow} or Corollary \ref{CorCex}, it yields a family of Riemannian manifolds that satisfy the estimates \eqref{DVonM} and \eqref{UEonM} for some $\beta \not\equiv 2$.

\bibliographystyle{alpha}

\bibliography{../Biblio}

\newcommand{\etalchar}[1]{$^{#1}$}
\begin{thebibliography}{HLM{\etalchar{+}}11}

\bibitem[AMM15]{AMM}
P.~Auscher, A.~McIntosh, and A.~J. Morris.
\newblock Calder{\`o}n reproducing formulas and applications to {H}ardy spaces.
\newblock {\em Rev. Mat. Iberoamericana}, 31:865--900, 2015.

\bibitem[AMR08]{AMR}
P.~Auscher, A.~McIntosh, and E.~Russ.
\newblock Hardy spaces of differential forms and {R}iesz transforms on
  {R}iemannian manifolds.
\newblock {\em J. Geom. Anal.}, 18(1):192--248, 2008.

\bibitem[Aus07]{Auscher2007}
P.~Auscher.
\newblock On necessary and sufficient conditions for ${L}^p$-estimates of
  {R}iesz transforms associated to elliptic operators on ${R}^n$ and related
  estimates.
\newblock {\em Mem. Amer. Math. Soc.}, 186(871):75 pp, 2007.

\bibitem[Bar98]{BarlowStFlour}
M.~T. Barlow.
\newblock Diffusions on fractals.
\newblock In {\em Lectures on probability theory and statistics
  ({S}aint-{F}lour, 1995)}, number 1690 in Lectures Notes in Math., pages
  1--121. Springer-Verlag, Berlin, 1998.

\bibitem[Bar04]{BarlowWV}
M.~T. Barlow.
\newblock Which values of the volume growth and escape time exponent are
  possible for a graph?
\newblock {\em Rev. Mat. Iberoamericana}, 20(1):1--31, 2004.

\bibitem[BCF14]{BCF}
F.~Bernicot, T.~Coulhon, and D.~Frey.
\newblock Gradient estimates, {P}oincar\'e inequalities, {D}e {G}iorgi property
  and their consequences.
\newblock 2014.
\newblock Avaiable at http://arxiv.org/pdf/1407.3906.pdf.

\bibitem[BCK05]{BCK05}
M.~T. Barlow, T.~Coulhon, and T.~Kumagai.
\newblock Characterization of sub-{G}aussian heat kernel estimates on strongly
  recurrent graphs.
\newblock {\em Comm. Pure Appl. Math.}, 58(12):1642--1677, 2005.

\bibitem[BGK12]{BGK}
M.~T. Barlow, A.~Grigor'yan, and T.~Kumagai.
\newblock On the equivalence of parabolic {H}arnack inequalities and heat
  kernel estimates.
\newblock {\em J. Math. Soc. Japan}, 64(4):1091--1146, 2012.

\bibitem[BK03]{BK}
S.~Blunck and P.~C. Kunstmann.
\newblock Calder\'on-{Z}ygmund theory for non-integral operators and the
  {$H^\infty$} functional calculus.
\newblock {\em Rev. Mat. Iberoamericana}, 19(3):919--942, 2003.

\bibitem[BR09]{BRuss}
N.~Badr and E.~Russ.
\newblock Interpolation of {S}obolev spaces, {L}ittlewood-{P}aley inequalities
  and {R}iesz transforms on graphs.
\newblock {\em Publ. Mat.}, 53:273--328, 2009.

\bibitem[BZ08]{BZ}
F.~Bernicot and J.~Zhao.
\newblock New abstract {H}ardy spaces.
\newblock {\em J. Funct. Anal.}, 255:1761--1796, 2008.

\bibitem[CCFR15]{CCFR}
L.~Chen, T.~Coulhon, J.~Feneuil, and E.~Russ.
\newblock Riesz transform for $1\leq p \leq 2$ without gaussian heat kernel
  bound.
\newblock {\em Available at http://arxiv.org/abs/1510.08275}, 2015.

\bibitem[CD99]{CDleq2}
T.~Coulhon and X.~T. Duong.
\newblock Riesz transforms for {$1\leq p\leq 2$}.
\newblock {\em Trans. Amer. Math. Soc.}, 351(3):1151--1169, 1999.

\bibitem[CD03]{CDgeq2}
T.~Coulhon and X.~T. Duong.
\newblock Riesz transform and related inequalities on noncompact {R}iemannian
  manifolds.
\newblock {\em Comm. Pure Appl. Math.}, 56(12):1728--1751, 2003.

\bibitem[CG98]{CoulGrigor}
T.~Coulhon and A.~Grigor'yan.
\newblock Random walks on graphs with regular volume growth.
\newblock {\em Geom. Funct. Anal.}, 8(4):656--701, 1998.

\bibitem[Che14]{ChenThesis}
Li~Chen.
\newblock {\em Quasi Riesz transforms, Hardy spaces and generalized
  sub-Gaussian heat kernel estimates}.
\newblock PhD thesis, Universit\'e Paris Sud - Paris XI; Australian national
  university, 2014.

\bibitem[CMS85]{CMS}
R.~R. Coifman, Y.~Meyer, and E.~M. Stein.
\newblock Some new function spaces and their applications to harmonic analysis.
\newblock {\em J. Funct. Analysis}, 62:304--335, 1985.

\bibitem[CSC90]{CSC}
T.~Coulhon and L.~Saloff-Coste.
\newblock Puissances d'un op\'erateur r\'egularisant.
\newblock {\em Ann. Inst. H. Poincar\'e Probab. Statist.}, 26(3):419--436,
  1990.

\bibitem[Del99]{Delmotte1}
T.~Delmotte.
\newblock Parabolic {H}arnack inequality and estimates of {M}arkov chains on
  graphs.
\newblock {\em Rev. Mat. Iberoamericana}, 15(1):181--232, 1999.

\bibitem[Dun06]{Dungey}
N.~Dungey.
\newblock A note on time regularity for discrete time heat kernels.
\newblock {\em Semigroups forum}, 72(3):404--410, 2006.

\bibitem[Dun08]{Dungey2}
N.~Dungey.
\newblock A {L}ittlewood-{P}aley-{S}tein estimate on graphs and groups.
\newblock {\em Studia Mathematica}, 189(2):113--129, 2008.

\bibitem[DY05]{DY2}
X.~T. Duong and L.~X. Yan.
\newblock Duality of {H}ardy and {BMO} spaces associated with operators with
  heat kernel bounds.
\newblock {\em J. Amer. Math. Soc.}, 18(4):943--973, 2005.

\bibitem[Fen]{Fen2}
J.~Feneuil.
\newblock Hardy and {BMO} spaces on graphs, application to {R}iesz transforms.
\newblock Available as http://arxiv.org/abs/ 1411.3352.

\bibitem[Fen15a]{FenThesis}
J.~Feneuil.
\newblock {\em Analyse harmonique sur les graphes et les groupes de Lie:
  fonctionnelles quadratiques, transformées de Riesz et espaces de Besov}.
\newblock PhD thesis, Universit\'e Grenoble Alpes, 2015.

\bibitem[Fen15b]{Fen1}
J.~Feneuil.
\newblock {L}ittlewood-{P}aley functionals on graphs.
\newblock {\em Math. Nachr.}, 288(11-12):1254--1285, 2015.

\bibitem[Fen15c]{FenLB}
J.~Feneuil.
\newblock On diagonal lower bound of markov kernel from ${L}^2$ analyticity.
\newblock {\em Available as http://arxiv.org/abs/1508.02447}, 2015.

\bibitem[FS72]{FS2}
C.~Fefferman and E.~M. Stein.
\newblock {$H^{p}$} spaces of several variables.
\newblock {\em Acta Math.}, 129(3-4):137--193, 1972.

\bibitem[Gri01]{grigofractal}
Alexander Grigor'yan.
\newblock Heat kernels on manifolds, graphs and fractals.
\newblock In {\em European {C}ongress of {M}athematics, {V}ol. {I}
  ({B}arcelona, 2000)}, volume 201 of {\em Progr. Math.}, pages 393--406.
  Birkh\"auser, Basel, 2001.

\bibitem[GT01]{GrigorT}
A.~Grigor'yan and A.~Telcs.
\newblock Sub-{G}aussian estimates of heat kernels on infinite graphs.
\newblock {\em Duke Math. J.}, 109(3):451--510, 2001.

\bibitem[HK04a]{HK04b}
B.~M. Hambly and T.~Kumagai.
\newblock Heat kernel estimates and law of the iterated logarithm for symmetric
  random walks on fractal graphs.
\newblock In {\em Discrete geometric analysis}, volume 347 of {\em Contemp.
  Math.}, pages 153--172. Amer. Math. Soc., Providence, RI, 2004.

\bibitem[HK04b]{HK04}
B.~M. Hambly and T.~Kumagai.
\newblock Heat kernel estimates for symmetric random walks on a class of
  fractal graphs and stability under rough isometries.
\newblock In {\em Fractal geometry and applications: a jubilee of {B}eno\^\i t
  {M}andelbrot, {P}art 2}, volume~72 of {\em Proc. Sympos. Pure Math.}, pages
  233--259. Amer. Math. Soc., Providence, RI, 2004.

\bibitem[HLM{\etalchar{+}}11]{HLMMY}
S.~Hofmann, G.~Lu, D.~Mitrea, M.~Mitrea, and L.~Yan.
\newblock Hardy spaces associated to non-negative self-adjoint operators
  satisfying {D}avies-{G}affney estimates.
\newblock {\em Mem. Amer. Math. Soc.}, 214(1007):vi+78, 2011.

\bibitem[HM09]{HoMay}
S.~Hofmann and S.~Mayboroda.
\newblock Hardy and {BMO} spaces associated to divergence form elliptic
  operators.
\newblock {\em Math. Ann.}, 344:37--116, 2009.

\bibitem[HSC93]{HebSC}
W.~Hebisch and L.~Saloff-Coste.
\newblock Estimates for {M}arkov chains and random walks on groups.
\newblock {\em The Annals of Probability}, 21(2):673--709, 1993.

\bibitem[KU15]{KU}
P.~C. Kunstmann and M.~Uhl.
\newblock Spectral multiplier theorems of {H}\"ormander type on {H}ardy and
  {L}ebesgue spaces.
\newblock {\em J. Operator Theory}, 73(1):27--69, 2015.

\bibitem[LMX12]{LMX}
C.~Le~Merdy and Q.~Xu.
\newblock Maximal theorems and square functions for analytic operators on
  {$L^p$}-spaces.
\newblock {\em J. Lond. Math. Soc. (2)}, 86(2):343--365, 2012.

\bibitem[Rus00]{Russ}
E.~Russ.
\newblock Riesz tranforms on graphs for $1\leq p\leq 2$.
\newblock {\em Math. Scand.}, 87(1):133--160, 2000.

\bibitem[Rus07]{Russ2}
E.~Russ.
\newblock The atomic decomposition for tent spaces on spaces of homogeneous
  type.
\newblock In {\em C{MA}/{AMSI} {R}esearch {S}ymposium ``{A}symptotic
  {G}eometric {A}nalysis, {H}armonic {A}nalysis, and {R}elated {T}opics''},
  volume~42 of {\em Proc. Centre Math. Appl. Austral. Nat. Univ.}, pages
  125--135. Austral. Nat. Univ., Canberra, 2007.

\bibitem[Ste70a]{Steinsingular}
E.~M. Stein.
\newblock {\em Singular integrals and differentiability properties of
  functions}.
\newblock Princeton Mathematical Series, No. 30. Princeton University Press,
  Princeton, N.J., 1970.

\bibitem[Ste70b]{Stein1}
E.~M. Stein.
\newblock {\em Topics in harmonic analysis related to the {L}ittlewood-{P}aley
  theory}.
\newblock Ann. of Math. Studies. Princeton University Press, Princeton, NJ,
  1970.

\end{thebibliography}

\end{document}